\documentclass[12pt]{amsart}

\usepackage[T1]{fontenc}
\usepackage{lmodern}
\usepackage{graphicx}
\usepackage{amssymb}

\usepackage{amsmath}
\usepackage{amsthm}

\usepackage[top=2cm,bottom= 3cm,left = 2cm,right = 2cm]{geometry}
\usepackage[utf8]{inputenc}

\usepackage{amsfonts}
\usepackage{geometry}

\usepackage{graphicx}
\usepackage{xcolor}
\usepackage{graphics}
\usepackage{listings}

\usepackage[english]{babel}

\usepackage{amsthm}

\theoremstyle{plain}
\newtheorem{theorem}{Theorem}[section]   
\newtheorem{lemma}[theorem]{Lemma}       
\newtheorem{proposition}[theorem]{Proposition}
\newtheorem{corollary}[theorem]{Corollary}

\theoremstyle{definition}

\theoremstyle{remark}
\newtheorem{remark}[theorem]{Remark}

\renewcommand{\epsilon}{\varepsilon}
\usepackage{amsmath}
\usepackage{graphicx}
\usepackage[colorlinks=true, allcolors=blue]{hyperref}

\begin{document}

\title[Multi-peaks in the zero-mass case]{Existence of multi-peak solutions for singularly perturbed
problems in the zero-mass case}

\author{Léo Demuth}
\address{DER de mathématiques, ENS Paris-Saclay, Gif-sur-Yvette, France}
\email{leo.demuth@ens-paris-saclay.fr}



\begin{abstract}
We prove the existence of multi-peak solutions for singularly perturbed elliptic
problems in the so-called zero-mass case by applying the classical
Lyapunov–Schmidt reduction method. The goal is to construct solutions concentrating around
several distinct interior points of the domain, identified as nondegenerate critical points of a potential. 
\end{abstract}

\maketitle
\markboth{Léo Demuth}{Multi-peaks in the zero-mass case}

\setcounter{tocdepth}{2}
\tableofcontents

\begin{center}
\section{Introduction}
\end{center}

The study of concentrating solutions for singularly perturbed nonlinear
Schr\"odinger equations has attracted considerable attention over the last
decades. A classical model is
\begin{equation}\label{classicalNLS}
 -\epsilon^2\Delta u+V(x)u=u^p,\qquad
 u>0\quad\text{in }\mathbb{R}^n,
\end{equation}
where \(1<p<(n+2)/(n-2)\) and \(\epsilon>0\) is a small parameter.
In the semiclassical regime \(\epsilon\to0\), one is interested in
solutions which develop one or more spikes whose location is determined by
the geometry of the potential \(V\).

The first results in this direction go back to Floer and Weinstein
\cite{FloerWeinstein1986}, who constructed, in the one-dimensional cubic
case, standing waves concentrating near nondegenerate critical points of
the potential. Their approach was subsequently extended to higher
dimensions and more general nonlinearities by Oh \cite{Oh1988}. Since
then, a large literature has developed around the existence and
multiplicity of semiclassical states. In particular, perturbative methods
were used by Ambrosetti, Badiale and Cingolani
\cite{AmbrosettiBadialeCingolani1997} to construct solutions concentrating
near critical points of \(V\), while variational and penalization methods,
notably those developed by del Pino and Felmer \cite{delPinoFelmer1998},
led to the construction of multi-peak solutions under more general
assumptions on the potential.

A fundamental feature of \eqref{classicalNLS} is the presence of a positive
mass term. Indeed, freezing the potential at a point \(P\), the limiting
profile satisfies
\begin{equation}\label{positiveMassLimit}
 -\Delta U+V(P)U=U^p
 \qquad\text{in }\mathbb{R}^n.
\end{equation}
When \(V(P)>0\), positive solutions of \eqref{positiveMassLimit} decay
exponentially at infinity. This exponential decay plays a crucial role in
the semiclassical analysis: in a multi-peak configuration, profiles whose
centers are separated at the original scale have exponentially small
interactions after rescaling.

In this paper we consider the singularly perturbed problem
\begin{equation}\label{Main}
 \tag{$3$}
 \begin{cases}
 -\epsilon^2\Delta u=u^p-Q(x)u^q & \text{in }\Omega,\\
 u>0 & \text{in }\Omega,\\
 u=0 & \text{on }\partial\Omega,
 \end{cases}
\end{equation}
where \(\Omega\) is either \(\mathbb{R}^n\) or a smooth bounded domain,
\(1<q<p<(n+2)/(n-2)\), and \(Q\) is positive and sufficiently regular.
Our main interest is the whole-space problem \(\Omega=\mathbb{R}^n\).

Problem \eqref{Main} may be regarded as a zero-mass counterpart of the
classical semiclassical Schr\"odinger equation. Indeed, when \(q=1\),
\eqref{Main} becomes
\[
 -\epsilon^2\Delta u+Q(x)u=u^p,
\]
and hence falls precisely within the classical positive-mass framework.
On the other hand, for \(q>1\), the derivative at the origin of the
nonlinearity
\[
 f(x,u)=u^p-Q(x)u^q
\]
vanishes. The corresponding frozen equation is
\begin{equation}\label{BVPinf}\tag{$4$}
 \begin{cases}
 -\Delta U=U^p-Q(P)U^q & \text{in }\mathbb{R}^n,\\
 U>0 & \text{in }\mathbb{R}^n,\\
 U(x)\to0 & \text{as }|x|\to\infty,
 \end{cases}
\end{equation}
and its positive radial solution \(U_P\) has only algebraic decay. In the
range relevant to our construction,
\[
 U_P(x)=O(|x|^{2-n})
 \qquad\text{as }|x|\to\infty.
\]
This difference has important consequences for the construction of
concentrating solutions.

Singularly perturbed problems of zero-mass type were previously
investigated by Dancer and Santra \cite{DancerSantra2010}. In particular,
for the analogue of \eqref{Main} in a bounded domain, they obtained
positive multi-spike solutions concentrating near prescribed isolated
local minima of \(Q\). Their analysis also highlights the special role
played by the exponent
\[
 q^*=\frac{n}{n-2}
\]
in the asymptotic behavior of the spikes.

The purpose of the present paper is to develop a perturbative theory for
\eqref{Main} in the whole space and to show that, in analogy with the
classical positive-mass Schr\"odinger problem, concentration is not
restricted to local minima of the potential. More precisely, we prove the
existence of positive multi-peak solutions concentrating near any
prescribed finite collection of distinct nondegenerate critical points of
\(Q\). Thus, in particular, the concentration points may be local maxima
or saddle points.

We now state the main result in the whole space. Set
\[
 q_1=\frac{(n+1)+\sqrt{(n-1)^2+8}}{2(n-2)}.
\]

\begin{theorem}\label{MainWholeSpace}
Assume that \(n\geq9\),
\[
 q_1<q<p<\frac{n+2}{n-2},
\]
and that \(Q\in C^2(\mathbb{R}^n)\) is positive, with \(Q\), \(\nabla Q\),
and \(\nabla^2Q\) bounded. Let
\(\xi_1,\ldots,\xi_K\) be distinct nondegenerate critical points of \(Q\).
Then, for all sufficiently small \(\epsilon>0\), problem
\eqref{Main} with \(\Omega=\mathbb{R}^n\) admits a positive solution
\(u_\epsilon\) and there exist points
\(P_{\epsilon,1},\ldots,P_{\epsilon,K}\) such that
\[
 P_{\epsilon,k}\longrightarrow\xi_k,
 \qquad
 u_\epsilon(P_{\epsilon,k})
 \longrightarrow U_{\xi_k}(0)>0,
 \qquad k=1,\ldots,K.
\]
Moreover,
\[
 \lim_{R\to\infty}\,
 \limsup_{\epsilon\to0}\,
 \sup_{\substack{x\in\mathbb{R}^n\\
 \operatorname{dist}(x,\{P_{\epsilon,1},\ldots,P_{\epsilon,K}\})
 \geq R\epsilon}}
 u_\epsilon(x)=0.
\]
\end{theorem}

Thus the solutions given by Theorem~\ref{MainWholeSpace} develop the
prescribed family of peaks, each localized at the scale \(\epsilon\) near
a nondegenerate critical point of \(Q\). To the best of our knowledge,
multi-peak concentration at arbitrary prescribed nondegenerate critical
points has not previously been established for this zero-mass problem in
the whole space.

Although the general perturbative philosophy is reminiscent of the
classical semiclassical theory, the zero-mass character of
\eqref{Main} introduces substantial analytical difficulties. The first
one is the algebraic decay of the limiting profiles. Peaks centered at
mutual distances of order \(1/\epsilon\) interact only polynomially,
rather than exponentially, and therefore the interaction terms require
considerably more delicate estimates. A second difficulty concerns the
functional setting. Since the limiting profiles do not in general belong
to \(L^2(\mathbb{R}^n)\), the natural framework is no longer
\(H^1(\mathbb{R}^n)\). We work instead in a suitable intersection space
of the form
\[
 X=D^{1,2}(\mathbb{R}^n)\cap L^s(\mathbb{R}^n),
\]
where \(s\) has to be chosen compatibly with the decay of the limiting
solution and with the nonlinear estimates required in the reduction.
These features are also responsible for some of the restrictions on the
dimension and on the exponent \(q\) appearing in our main theorem.

Our proof is based on a Lyapunov--Schmidt reduction. For a configuration
\(P=(P_1,\ldots,P_K)\), we start from a superposition of translated
solutions \(U_{P_k}\) of the frozen problem \eqref{BVPinf}. The
nondegeneracy of the single-peak profile allows us to solve the
infinite-dimensional auxiliary equation in the orthogonal complement of
the approximate kernel. A central point is to obtain an invertibility
estimate which is uniform both with respect to \(\epsilon\) and to the
position of the peaks. Because of the algebraic tails of the profiles,
this step requires compactness and interaction estimates adapted to the
space \(X\).

The remaining finite-dimensional equations determine the location of the
peaks. Their leading-order term is proportional to
\[
 \epsilon\,\nabla Q(P_k),
\]
so that, after normalization, the reduced problem converges to the system
\[
 \nabla Q(P_1)=\cdots=\nabla Q(P_K)=0.
\]
The nondegeneracy of the prescribed critical points then allows us to
conclude by a degree argument.

For completeness, we also consider the corresponding Dirichlet problem in
a bounded smooth domain. In that setting the whole-space profiles have to
be projected onto the rescaled domain, and additional boundary errors
arise. Since the concentration points remain uniformly away from the
boundary, these corrections are of lower order and the whole-space
reduction can be suitably adapted. This gives an analogous concentration
result at nondegenerate critical points of \(Q\). We regard this result
mainly as an extension of the whole-space construction and of the
bounded-domain analysis developed in the zero-mass setting by Dancer and
Santra \cite{DancerSantra2010}.

Several natural questions remain open. First, our argument requires the
dimensional restriction \(n\geq9\). This condition enters through the
estimates of the algebraic interactions and of the error terms, and it
would be interesting to understand whether it is merely technical or
whether a different phenomenon occurs in lower dimensions. Second, the
present construction requires \(q\) to lie above the threshold
\(q_1>n/(n-2)\). It is natural to ask whether concentration near arbitrary
nondegenerate critical points persists throughout the larger range
\[
 \frac{n}{n-2}<q<p.
\]
Finally, in the classical positive-mass problem, semiclassical
concentration can be obtained under substantially weaker geometric
assumptions than nondegeneracy of an isolated critical point. It would
therefore be interesting to investigate the zero-mass problem in the
presence of degenerate critical points or critical manifolds of \(Q\).

The paper is organized as follows. Section 2 is devoted to the limiting
zero-mass problem and collects the properties of its positive radial
solution and of the corresponding linearized operator that are needed in
the sequel. Section 3 contains the main part of the analysis and treats
the whole-space problem: we introduce the functional framework, carry out
the Lyapunov--Schmidt reduction, solve the auxiliary and
finite-dimensional equations, and finally establish the concentration
properties of the solutions. Section 4 adapts the construction to a
smooth bounded domain.

$\\$

In this paper, we adopt the following conventions:
\begin{itemize}
    \item \(C,C'\), or \(C''\) denote positive constants which may change
    from line to line and are independent of \(P\) and \(\epsilon\), unless
    otherwise specified. All \(o(\cdot)\) and \(O(\cdot)\) estimates are
    uniform with respect to \(P\).
    \item The letter \(k\) is used for a general peak, while \(k_0\)
    denotes a fixed peak.
    \item The letter \(m\) is used to index sequences.
    \item The letter \(j\) denotes a coordinate index.
\end{itemize}
\begin{remark}
     AI-based language assistance was used to improve the English grammar and style of this report.
It was not used to generate mathematical arguments, proofs or results.

\end{remark}

\paragraph{Acknowledgments.}
The author would like to thank Professor Angela Pistoia for suggesting the
problem and for many helpful discussions. The author is also grateful to
the Department of Basic and Applied Sciences for Engineering (SBAI) of
Sapienza University of Rome for its hospitality.

\begin{center}
\section{The limit problem}
\end{center}

Throughout the rest of the paper we assume that
\[
 n\geq 9,\qquad
 q_1<q<p<\frac{n+2}{n-2},
 \qquad
 q_1=\frac{(n+1)+\sqrt{(n-1)^2+8}}{2(n-2)},
\]
and that \(Q\) is positive and of class \(C^2\), with
\(Q\), \(\nabla Q\), and \(\nabla^2Q\) bounded. When \(\Omega\) is a
bounded domain, these assumptions are understood on \(\overline\Omega\).

For \(P\in\Omega\), after the change of variables \(x=P+\epsilon y\),
problem \eqref{Main} is rewritten as
\begin{equation}\label{BVP1}\tag{$5$}
 \begin{cases}
 -\Delta u(y)=u^p(y)-Q(P+\epsilon y)u^q(y)
     & \text{in }\Omega_{\epsilon,P},\\
 u>0 & \text{in }\Omega_{\epsilon,P},\\
 u=0 & \text{on }\partial\Omega_{\epsilon,P},
 \end{cases}
\end{equation}
where
\[
 \Omega_{\epsilon,P}
 =
 \{y\in\mathbb{R}^n:\ P+\epsilon y\in\Omega\}.
\]
For \(\Omega=\mathbb{R}^n\), of course,
\(\Omega_{\epsilon,P}=\mathbb{R}^n\).

The limit equation \eqref{BVPinf} plays an important role in this article. Thanks to \cite{LiNi1993,KwongZhang1991,BerestyckiLions1983}, we know  the problem 

\begin{equation*}\label{BVP0} 
 \begin{cases} 
 -\Delta u=u^p-u^q\quad  \text{in} \quad \mathbb{R}^n, \\
        \quad u(x) \rightarrow0, \quad |x|\rightarrow +\infty,\\

\end{cases}
\end{equation*}

\noindent has a unique positive radial  solution that we denote $U$. Thanks to these articles, we also know that $U$ is radially decreasing and smooth.
It is immediate to check  
$$U_{P}(x) = Q(P)^{\frac{1}{p-q}}U\left(Q(P)^{\frac{p-1}{2(p-q)}}x\right)$$ solves

 \begin{equation*}\label{BVP0} 
 \begin{cases} 
 -\Delta u=u^p-Q(P)u^q\quad  \text{in} \quad \mathbb{R}^n,  \\
        \quad u(x) \to 0 & \text{as } |x| \to +\infty.\\

\end{cases}
\end{equation*} 
For simplicity we denote $a=\frac{1}{p-q}$ and $b=\frac{p-1}{2(p-q)}$ so that
$$U_{P}(x) = Q(P)^{a}U\left(Q(P)^bx\right).$$

\noindent In addition, the function $U$ and its derivatives decrease fast enough at infinity to allow us to use our method. 

Indeed, it was proved in \cite{DancerSantra2010} that because $q>q^*=\frac{n}{n-2}$, we have  
 
$$ U(|x|) = O\left(\frac{1}{|x|^\alpha}\right) $$
 where $\alpha = n-2$.
 
\noindent Thus, there exists a constant $C$ (which depends on $q$), such that
$$ 0\leq U(x) \leq \frac{C}{1+|x|^{\alpha}}$$

In \cite{DancerSantra2010}, the authors also prove that 
$$U_r = O\left(\frac{1}{1+r^{\alpha+1}}\right)$$
where $U_r$ is the radial derivative of the function $U$. Using the equation defining $U$ and the expression of the Laplacian in polar coordinates, we deduce the similar inequalities on the higher order derivatives. 
\medskip

 Finally, we can associate to Equation \eqref{BVPinf} the natural linearized operator $L_P:Z\mapsto -\Delta Z-(pU_P^{p-1}Z-qQ(P)U_P^{q-1}Z)$. During our work we will also associate a natural linearized operator $L_{\epsilon,P}$ to Equation \eqref{BVP1} and will in some sense show that $L_{\epsilon,P}$ is near to $L_P$ when $\epsilon$ tends to $0$. Then the study of $L_P$ plays an important role and the following result proved for example in \cite{DancerSantra2010} is fundamental. 
$\\$

\begin{lemma}\label{non-degeneracy of the solution}
Consider $P\in \mathbb{R}^n$. The solutions in $D^{1,2}(\mathbb{R}^n)=\left\{u \in L^{2^*}(\mathbb{R}^n) | \nabla u \in L^2(\mathbb{R}^n) \right \}$ of the linearized problem 
\begin{equation*}\label{linear} 
   -\Delta Z=pU_P^{p-1}Z-qQ(P)U_P^{q-1}Z\quad  \text{in} \quad \mathbb{R}^n.
\end{equation*} 
are linear combinations of 
$\partial U_P\over\partial x_j$,  $j=1,\dots,n$. 
\end{lemma}
\begin{center}
\section{Case $\Omega = \mathbb{R}^n$}

\end{center}

\noindent In this section, we suppose $\Omega=\mathbb{R}^n$. In this case, Equation \eqref{Main} becomes
\begin{equation}\label{Main-unbounded} 
 \tag{$6$}
 \begin{cases} 
    -\epsilon^2\Delta u(x)=u^p(x)-Q(x)u^q(x), \quad &\text{in} \quad \mathbb{R}^n\\
        \quad u > 0  \quad &\text{in}  \quad \mathbb{R}^n. \\
       \quad u(x) \to 0 & \text{as } |x| \to +\infty.
\end{cases}
\end{equation}

\subsection{Choice of the space}

 In our study, we will work with the space $$D^{1,2}(\mathbb{R}^n) = \left\{u \in L^{2^*}(\mathbb{R}^n) | \nabla u \in L^2(\mathbb{R}^n) \right \}.$$
$\\$ We can define a scalar product on $D^{1,2}(\mathbb{R}^n)$ as follows: 
$\forall u,v \in D^{1,2}(\mathbb{R}^n),\langle u,v\rangle =\int_{\mathbb{R}^n}\nabla u(x)\nabla v(x) dx$.
Recall that we have thanks to the Gagliardo-Nirenberg-Sobolev inequality for $u \in D^{1,2}(\mathbb{R}^n)$

$$\|u\|_{L^{2^*}} \leq C_n\|\nabla u\|_{L^2(\mathbb{R}^n)}.$$

\noindent Then, the norm of the space $D^{1,2}(\mathbb{R}^n)$ can control the norm of $L^{\frac{2n}{n-2}} (\mathbb{R}^n)$. Nonetheless, we will need in our calculations to control some other norms $L^s$ that will not be possible with only the norm of  $D^{1,2}(\mathbb{R}^n)$ because $\mathbb{R}^n$ is not bounded. To solve this problem, we need to introduce a suitable $s$  and work on the space $L^s(\mathbb{R}^n)\cap D^{1,2}(\mathbb{R}^n)$.
The further calculations force us to choose $s \in \mathbb{R}$ such that $max\left(\frac{n}{n-2},\frac{n}{(n-2)q-3}\right )<s\leq min\left(\frac{n(q-1)}{2},\frac{2n}{n+2}\right )$. It is possible because 
\begin{itemize}
    \item  $q > q^*$ so that $\frac{n(q-1)}{2}>\frac{n}{n-2}$;
    \item $n\geq9$ so that $\frac{2n}{n+2} > \frac{n}{n-2}$;
    \item $q >q^*$ so $\frac{n}{(n-2)q-3}\leq \frac{n}{n-3} < \frac{2n}{n+2} $ because $n\geq9$;
    \item $\frac{n}{(n-2)q-3} < \frac{n(q-1)}{2}$. Indeed this inequality is equivalent to:
    $(n-2)q^2-(n+1)q+1>0$. But discriminant associated is $(n+1)^2-4(n-2)=(n-1)^2+8$ and the biggest root of this polynomial  is $\frac{(n+1)+\sqrt{(n-1)^2+8}}{2(n-2)}$.  One has by hypothesis 
    $q>q_1 =\frac{(n+1)+\sqrt{(n-1)^2+8}}{2(n-2)}.$

\end{itemize}
   $\\$

 In the rest of the section, we will work on $X= L^{s}(\mathbb{R}^n) \cap D^{1,2}(\mathbb{R}^n)$ associated with the norm
$\|u\|_X= max\left\{\|u\|_{D^{1,2}},\|u\|_{s}\right\}$.
It is clearly a Banach and reflexive space.
$\\$

One of the crucial points is that $U \in X$ (thus $U_P \in X$). Indeed,
\begin{itemize}
   \item In each case, $2(\alpha+1)>n$, because $2(\alpha+1) = 2(n-1)$ and $n>2$. Thus, $U \in D^{1,2}(\mathbb{R}^n)$.

 \item Then, $s$ is chosen so that $U \in L^s(\mathbb{R}^n)$ (because $s >  \frac{n}{n-2}$ so that $\alpha s >n$). 
\end{itemize}

Furthermore, we can use the properties of $X$ to  rewrite the partial differential equation with tools of functional analysis. To achieve this define the embedding $i: D^{1,2}(\mathbb{R}^n) \rightarrow L^{2^*}(\mathbb{R}^n)$ which is continuous thanks to the Gagliardo-Nirenberg-Sobolev inequality.

\medskip
By definition the adjoint $i^*:L^{\frac{2n}{n+2}}(\mathbb{R}^n)\rightarrow D^{1,2}(\mathbb{R}^n)$ of the operator $i$ respects

$$i^*(v)= u \iff \forall \phi \in D^{1,2}(\mathbb{R}^n), \langle u,\phi\rangle=\int_{\mathbb{R}^n} v\phi.$$ 
 $i^*$ is well defined thanks to the Riesz Theorem because $\phi \rightarrow \langle v,\phi\rangle_2$ is continuous. Indeed, by Hölder's inequality
$$|\langle v,\phi\rangle_2|\leq \|v\|_{L^{\frac{2n}{n+2}}}\|\phi\|_{L^{2^*}} \leq C\|\phi\|_{D^{1,2}}.$$

We will then use $i^*$ in the next sections to find a new formulation of our problem. Thus, we need to control $\|i^*(v)\|_X$ when $v$ is smooth enough. We need so to control $\|i^*(v)\|_{D^{1,2}(\mathbb{R}^n)}$ and $\|i^*(v)\|_{L^s(\mathbb{R}^n)}$. We can use the following facts:

\begin{itemize}
    \item $i^*$ is continuous and $$\|i^*(v)\|_{D^{1,2}(\mathbb{R}^n)}\leq C\|v\|_{L^{\frac{2n}{n+2}}(\mathbb{R}^n)}.$$ 
    \item To control $\|i^*(v)\|_s$, using Hardy-Littlewood-Sobolev inequality, we obtain
    \begin{lemma}\label{control L^s}

 If $v \in L^{\frac{2n}{n+2}}(\mathbb{R}^n)\cap L^{{\frac{ns}{n+2s}}}(\mathbb{R}^n)$ then $i^*(v) \in L^s(\mathbb{R}^n)$ and $\|i^*(v)\|_s \leq C\|v\|_{\frac{sn}{n+2s}}$.
\end{lemma}

\end{itemize}
\begin{remark}
    The conditions on $s$ might seem arbitrary. In order to understand the importance of these hypotheses in this article we can draw up the following list of conditions  where these inequalities appear:

    \begin{itemize}
    \item To use Hardy-Littlewood-Sobolev inequality as we did to prove \ref{control L^s}, we need $s>\frac{n}{n-2}$. 
        \item For $\phi \in X$, we need $\|\phi^{q-1}\|_\frac{n}{2} <+\infty$. Then by interpolation it suffices $s\leq\frac{n}{2}(q-1)\leq2^*.$
        \item  Several times we will use that $\beta(\alpha q-1)>n$ where $\beta = \frac{ns}{n+2s}$ (see for example Subsection 3.4 or Subsection 4.4). This inequality is equivalent to $s>\frac{n}{(n-2)q-3}$ (here we use that $q>q_1$).
        \item If $\phi \in X$, we need, $\phi^p \in L^{\frac{2n}{n+2}}$. We will check later that it suffices $s<\frac{2n}{n+2}$ by interpolation. 
    \end{itemize}
\end{remark}

\subsection{The existence of a multi peaks solution}

We now turn to the proof of Theorem~\ref{MainWholeSpace}. For clarity,
we first present the construction for \(K=2\). The extension to an
arbitrary fixed finite number of peaks follows from the same arguments
and is explained at the end of Section 3.5. In the case \(K=2\), the
statement reduces to
 \medskip
\medskip
\begin{theorem}\label{Main 2 peaks}
Suppose $q\in(q_1,p)$ and $n\geq9$.
Suppose that $\xi_1,\xi_2$ are $2$ different
non-degenerate critical points of $Q$. Then there exists $\epsilon_0 > 0$ such that, for all $0 < \epsilon < \epsilon_0$,
problem \eqref{Main} admits a positive solution $u_{\epsilon}$ which concentrates at $(\xi_k)_{k \in \{1,2\}}$.

\end{theorem}
\medskip 

\noindent Until the end we consider $\xi_1, \xi_2\in \mathbb{R}^n$  \noindent two non-degenerate critical points of $Q$. 

 \noindent  Recall that the equation \eqref{Main} can be written as

\begin{equation}\label{BVP1-R} 
 \tag{$7$}
 \begin{cases} 
    -\Delta u(x)=u^p(x)-Q(\epsilon x)u^q(x) \quad &\text{in} \quad \mathbb{R}^n, \\
        \quad u > 0  \quad &\text{in}  \quad \mathbb{R}^n,\\
        \quad u(x)\rightarrow0,|x|\rightarrow +\infty,
\end{cases}
\end{equation} 
We will rather study the problem
\begin{equation}\label{Alternative} 
 \tag{$8$}
 \begin{cases} 
    -\Delta u(x)=f(u)(x)-Q(\epsilon x)g(u)(x) \quad &\text{in} \quad \mathbb{R}^n, \\
    \quad u(x)\rightarrow0,|x|\rightarrow +\infty,\\
          \quad u \in X,
\end{cases}
\end{equation} 

with $f(x)=(x^+)^p$ and $g(x)=(x^+)^q$ ( see that the use of the positive part helps erasing the condition of positivity).
For a fixed $\epsilon > 0$, we expect to find a solution with the form $$W_{\epsilon,P}= \sum_{k=1}^2 U_{\epsilon,P_k} + \phi_{\epsilon,P} .$$ with $U_{\epsilon,P}=U_P(x-\frac{P}{\epsilon})$, $K_{\epsilon,P}=Span\{\frac{\partial U_{\epsilon,k}}{\partial x_j},k\in \{1,2\},j\in\{1,...,n\}\}$ and $\phi_{\epsilon,P}\in K_{\epsilon,P}^\perp=\{v \in X, \langle v|\frac{\partial U_{\epsilon,k}}{\partial x_j}\rangle_{D^{1,2}(\mathbb{R}^n)}=0, k\in \{1,2\} ,j\in  \{1,\ldots,n\}\}$, $P_{k} \in \mathbb{R}^n$ and  with $P\in \mathbb{R}^{2n}$ denoting the vector $(P_1,P_2)$. 
$\\$
In the following, we will write $U_{\epsilon,k}$ to mean $U_{\epsilon,P_k}$ and $S_{\epsilon,P}=\sum_{k=1}^2 U_{\epsilon,k}$ and we will consider $P \in \mathbb{R}^{2n}$.  
The equation \eqref{Alternative} becomes
 
\begin{equation}\label{Main-single}\tag{$9$}
W = i^*\bigl(f(W)-g(W)Q_\varepsilon\bigr)
\end{equation}

with $Q_\epsilon(x)=Q(\epsilon x)$.
This is well defined because 
\begin{itemize}
    \item $Q$ is bounded;
    \item $f(W) \in L^{\frac{2n}{n+2}}(\mathbb{R}^n)$. Indeed, one has $W \in L^{\frac{2n}{n-2}}$ and $W \in L^{s}$. Then by interpolation because $s \leq p\frac{2n} {n+2}\leq \frac{2n}{n-2}$, $W^p \in L^{\frac{2n}{n+2}}(\mathbb{R}^n)$. Indeed, $p \geq 1$ thus $ p\frac{2n}{n+2} \geq \frac{2n}{n+2}  \geq s$ and $\frac{n+2}{n-2} \geq p$ thus $\frac{2n}{n-2} \geq p\frac{2n}{n+2}$.
    \item By the same arguments, $Q_{\epsilon} W^q \in L^{\frac{2n}{n+2}}(\mathbb{R}^n)$ ( recall that $Q$ is bounded).
\end{itemize}
In order to find a strong solution of \eqref{Alternative} and then \eqref{BVP1-R}, we will use the classical Lyapunov-Schmidt reduction. We can sum up this method as follows

\begin{itemize}
    \item Denoting $\Pi_{\epsilon,P}:X\rightarrow X$ the orthogonal projection on $K_{\epsilon,P}^\perp$ in $D^{1,2}(\mathbb{R}^n)$, the equation \eqref{Main-single} can be divided in two parts 
    \begin{equation*}\label{BVP2-R} 
  \begin{cases} 
\Pi_{\epsilon,P}(W_{\epsilon,P} - i^*\bigl(f(W_{\epsilon,P})-g(W_{\epsilon,P})Q_\epsilon\bigr))=0    \quad \text{(the auxiliary equation)} \\
     \forall k \in \{1,2\},j \in \{1,...,n\},\left\langle 
    W_{\epsilon,P} - i^*((f(W_{\epsilon,P})-g(W_{\epsilon,P})Q_\epsilon)| \frac{\partial U_{\epsilon,k}}{\partial x_j}\right\rangle=0\quad\text{ (the bifurcation equation)} \\
   \end{cases} 
    \end{equation*}

\item Firstly using the definition of $K_{\epsilon,P}$, we rewrite the auxiliary equation as a fixed point problem $\phi=T_{\epsilon,P}\phi$  for all $P$ in good compact set $E$ and $\epsilon$ small enough. We can then use the theorem of Banach to solve a solution $W_{\epsilon,P} \in K_{\epsilon,P}^\perp$.
\item Then we can play with the parameter $P$ to solve the bifurcation equation. Using the theory of the degree of Brouwer, we can find a good sequence of $P_{\epsilon}$ such that $W_{\epsilon,P_{\epsilon}}$ solve the bifurcation equation and $P_{\epsilon,k}\rightarrow \xi_k$.

\item Finally we show that $W_{\epsilon,P_{\epsilon}}>0$, $W_{\epsilon,P_{\epsilon}}$ is regular enough, tends to $0$ at infinity and this solution concentrates.

\end{itemize}

\bigskip

\begin{remark}
Through this paper we will often use the following fact: suppose that $A$ belongs to a compact set $E$ and consider $\beta >0$ such that $U \in L^{\beta}(\mathbb{R}^n)$. Then $$||U_{\epsilon,A}||_{\beta}\leq Q(A)^{a-\frac{bn}{\beta}}||U||_\beta\leq C||U||_\beta$$ which $C$ depending only on $E$ because $U_{\epsilon,A}(x)=Q(A)^aU(Q(A)^b\left(x-\frac{A}{\epsilon})\right)$ and $Q(x)>0$.
\end{remark}

$\\$

\subsection{Solving the auxiliary equation}

\subsubsection{Decomposition of the equation}

We rewrite Equation \eqref{Main-single} as follows:

$$W=N_f(\phi) + N_g(\phi) + I_f + I_g+ I_{f'} + I_{g'} +\sum i^*(g'(U_{\epsilon,k})\phi(Q(P_k)-Q_{\epsilon}))+\sum i^*((f'(U_{\epsilon,k})-g'(U_{\epsilon,k})Q(P_k))\phi)+$$$$  \sum i^*(f(U_{\epsilon,k})-Q(P_k)g(U_{\epsilon,k}))+\sum i^*((Q(P_k)-Q_\epsilon) g(U_{\epsilon,k})) $$

$$\left\{
\begin{array}{ll}
      N_f(\phi)=i^*((f(S_{\epsilon,P}+\phi)-f(S_{\epsilon,P})-f'(S_{\epsilon,P})\phi))\\
        N_g(\phi) = -i^*(Q_{\epsilon}[g(S_{\epsilon,P}+\phi)-g(S_{\epsilon,P})-g'(S_{\epsilon,P})\phi])\\
        
       I_f= i^*(f(S_{\epsilon,P})- \sum f(U_{\epsilon,k}))\\ I_g= -i^*(Q_{\epsilon}(g(S_{\epsilon,P})- \sum g(U_{\epsilon,k})))\\
       I_{f'}=i^*((f'(S_{\epsilon,P}) - \sum f'(U_{\epsilon,k}))\phi) \\
       I_{g'}= -i^*((g'(S_{\epsilon,P}) - \sum g'(U_{\epsilon,k}))Q_{\epsilon}\phi)\\    
\end{array}
\right.  $$

$$$$

Hence, because $S_{\epsilon,P}= \sum i^*(f(U_{\epsilon,k})-Q(P_k)g(U_{\epsilon,k}))$ and $W_{\epsilon,P}=S_{\epsilon,P}+\phi_{\epsilon,P}$

$$\phi = N_f(\phi) + N_g(\phi) + I_f + I_g+ I_{f'} + I_{g'}  +\sum i^*(g'(U_{\epsilon,k})\phi(Q(P_k)-Q_{\epsilon}))+\sum i^*([f'(U_{\epsilon,k})-g'(U_{\epsilon,k})Q(P_k)]\phi)+ $$$$  \sum i^*((Q(P_k)-Q_{\epsilon})g(U_{\epsilon,k})). $$

$$$$
Set
$$L_{\epsilon,P}\phi = \phi - \sum i^*((f'(U_{\epsilon,k})-g'(U_{\epsilon,k})Q(P_k))\phi).$$

\noindent Then we have 

$$L_{\epsilon,P}\phi = N_f(\phi) + N_g(\phi)  + I_f + I_g + I_{f'}(\phi) + I_{g'}(\phi)+\sum i^*((Q(P_k)-Q_{\epsilon})g(U_{\epsilon,k})) $$$$+ \sum i^*(g'(U_{\epsilon,k})\phi(Q(P_k)-Q_{\epsilon})). $$

\subsubsection{Study of $L_{\epsilon,P}$}
Now, we study the linearized operator $L_{\epsilon,P}$ in order to simplify the previous equation and to see how to "invert" this operator. For this, we will show that $L_{\epsilon,P}$ is continuous and is a Fredholm operator.

\begin{lemma}\label{continuity 1}

$L_{\epsilon,P}\phi \in X, \forall \phi \in X$ and $L_{\epsilon,P}$ is continuous on $X$.
\end{lemma}

\begin{proof} 
Let $\phi \in X$. We want to show that $\|L_{\epsilon,P}(\phi)\|_s\leq C\|\phi\|_X$ and $\|L_{\epsilon,P}(\phi)\|_{D^{1,2}(\mathbb{R}^n)}\leq C'\|\phi\|_X$. Take $k\in \{1,2\}$. 
\begin{itemize}
\item One has by continuity of $i^*$ $$\|i^*((f'(U_{P_k})-Q(P_k) g'(U_{P_k}))\phi)\|_{D_{1,2}(\mathbb{R}^n)}  \leq \|(f'(U_{P_k})-Q(P_k) g'(U_{P_k}))\phi\|_{L^{\frac{2n}{n+2}}} \leq C \|\phi\|_X$$

because $U_{P_k}$ and $Q$ are bounded and $s \leq \frac{2n}{n+2} \leq \frac{2n}{n-2}$.
\item By Lemma \ref{control L^s} we obtain 
$$\|i^*((f'(U_{P_k})-Q(P_k) g'(U_{P_k}))\phi)\|_{s} \leq \|(f'(U_{P_k})-Q(P_k)g'(U_{P_k}))\phi\|_{\frac{ns}{n+2s}}.$$
$\\$
  By Hölder's inequality  (with $\frac{1}{\frac{sn}{n+2s}}=\frac{2}{n}+\frac{1}{s}$), we have

$$\|U_{P_k}^{q-1}Q(P_k)\phi\|_{\frac{ns}{n+2s}}\leq  \|\phi\|_s\|U_{P_k}^{q-1}\|_{\frac{n}{2}}$$

because $|Q(P_k)| \leq \|Q\|_{\infty}$.

But
 $\alpha(q-1)\frac{n}{2}=(n-2)(q-1)\frac{n}{2} >(n-2)(q^*-1)\frac{n}{2}=n$.

Then, because $U(r)=O\left(\frac{1}{r^{n-2}}\right)$ 

$$\|U_{P_k}^{q-1}Q(P_k)\|_\frac{n}{2}\leq C'.$$
Thus 

$$\|g'(U_{P_k})Q(P_k)\phi\|_{\frac{ns}{n+2s}} \leq  C' \|\phi\|_X.$$

Similarly
$$\|f'(U_{P_k})\phi\|_{\frac{ns}{n+2s}} \leq  C'' \|\phi\|_X.$$
\end{itemize}
Conclusion:
$$\|L_{\epsilon,P}\phi\|_X\leq \|\phi\|_X+\sum_k\|i^*((f'(U_{P_k})-Q(P_k) g'(U_{P_k}))\phi)\|_X\leq M\|\phi\|_X.$$

Thus, 
$L_{\epsilon,P}\phi \in X$. The previous calculations also show that $\|L_{\epsilon,P}\| < +\infty$.
\end{proof}

In addition, we have 

\begin{lemma}\label{Compact}
$\\$

 $Id-L_{\epsilon,P}$ is a compact operator.

\end{lemma}
\medskip
\begin{proof}
    
$\\$

Since the sum of compact operators is a compact operator, it suffices to show that
\begin{equation*}
\begin{cases}
X \rightarrow X \\
 \phi \mapsto i^*((f'(U_{P_k}) - g'(U_{P_k})Q(P_k))\phi)
\end{cases}
\end{equation*}

is compact for $k\in \{1,2\}$. We will show that $ \phi \mapsto i^*\left(g'(U_{P_k})\phi\right)$ is compact (because $Q(P_k)$ is constant and the further reasoning can be applied to show $ \phi \mapsto i^*(f'(U_{P_k})\phi)$ is compact).
It suffices to show that 
\begin{equation*}
\begin{cases}
T:X \rightarrow L^\frac{2n}{n+2}\left(\mathbb{R}^n\right)\cap L^{\frac{ns}{n+2s}}\left(\mathbb{R}^n\right) \\
 \quad \quad \phi \mapsto  g'(U_{P_k})\phi
\end{cases}
\end{equation*} is compact because $i^*:L^\frac{2n}{n+2}\left(\mathbb{R}^n\right)\cap L^{\frac{ns}{n+2s}}\left(\mathbb{R}^n\right) \rightarrow X$ is continuous. 

\noindent Set $T_m = 1_{B(0,m)}T$.
By the Rellich theorem, $T_m$ is compact. Indeed,
Let $(\phi_j)$ be bounded in $X$. By Hölder's inequality
$\big(\tfrac12=\tfrac1{2^*}+\tfrac1n\big)$ we have
$\|\phi_j\|_{L^2(B(0,m))}\le |B(0,m)|^{1/n}\|\phi_j\|_{2^*}$, so $(\phi_j)$ is
bounded in $H^1(B(0,m))$; by Rellich's theorem, up to a subsequence
$\phi_j\to\phi$ in $L^r(B(0,m))$ for every $r<2^*$. Since
$\frac{ns}{n+2s}<\frac{2n}{n+2}<2<2^*$, this applies to both exponents, and as
$g'(U_{P_k})\in L^\infty(\mathbb{R}^n)$,
\[
\|T_m\phi_j-T_m\phi\|_{\beta}\le \|g'(U_{P_k})\|_\infty\,
\|\phi_j-\phi\|_{L^{\beta}(B(0,m))}\longrightarrow 0,
\qquad \beta\in\Big\{\tfrac{2n}{n+2},\ \tfrac{ns}{n+2s}\Big\}.
\]
Hence $T_m$ is compact.
\smallskip
However, one has $$\|T-T_m\|_{X} \rightarrow 0.$$ Indeed, one has for $\phi \in X$
$$\|(T-T_m)\phi\|^{\frac{2n}{n+2}}_{\frac{2n}{n+2}}=\int_{B(0,m)^c}|g'(U_{P_k})\phi|^{\frac{2n}{n+2}} \leq g'(U_{P_k}(m))^{\frac{2n}{n+2}} \int_{B(0,m)^c}|\phi|^{\frac{2n}{n+2}}\leq C \frac{1}{m^{\alpha(q-1)\frac{2n}{n+2}}} \|\phi\|_X^{\frac{2n}{n+2}}$$
because $U_{P_k}$ is decreasing. Thus,
$$\|T-T_m\|_{\frac{2n}{n+2}} \rightarrow 0.$$

In addition, by Hölder's inequality, 
$$\|(T-T_m)\phi\|_{\frac{ns}{n+2s}}\leq \|\phi\|_{s}\|U_{P_k}^{q-1}\|_{L^{\frac{n}{2}}(B(0,m)^c)}\leq C\|\phi\|_X\frac{1}{m^{{(q-1)(n-2)}-2}}$$
 
which concludes the proof because ${(q-1)(n-2)-2} >0$ .
$\\$$\\$
\end{proof}

We need the last lemma 
\begin{lemma}\label{ContinuityP}
$\\$

Consider a compact set $E$. Then $P\in E \rightarrow  L_{P} \in L(X)$ is continuous where $L_P\phi = \phi -i^*((f'(U_P)-g'(U_P)Q(P))\phi)$. 
$$\\$$
\end{lemma}
\begin{proof}

Set $\nu:=inf_{x \in E}Q(x)$.
Let $P_m \rightarrow P$. Take $\phi \in X$. One has,

$$\|(L_{P}-L_{P_m})\phi\|_X \leq \|i^*((f'(U_{P})-f'(U_{P_m}))\phi)\|_X+ \|i^*((Q(P)-Q(P_m))g'(U_P)\phi)\|_X$$$$+ \|i^*((Q(P_m)(g'(U_P)-g'(U_{P_m}))\phi)\|_X.$$
 Take a real $t$ such that $ \frac{2n}{n+2}<t < 2^*$. By
 Hölder's inequality applied with the decomposition $\frac{1}{\frac{2n}{n+2}}=\frac{1}{t}+\frac{1}{\frac{t2n}{(n+2)t-2n}}$, we obtain  (possible because $t > \frac{2n}{n+2}$ then $\frac{t2n}{(n+2)t-2n}>0$) 
$$\|(Q(P)(g'(U_P)-g'(U_{P_m}))\phi\|_{\frac{2n}{n+2}}\leq \| \phi\|_{t}\|g'(U_P)-g'(U_{P_m})\|_l$$ 
with $l= \frac{t2n}{(n+2)t-2n}$ such that $U^{(q-1)l} \in L^1$. Indeed,
$\alpha(q-1)\frac{t2n}{(n+2)t-2n} >n $ because $t <2^*$ thus $2\frac{t2n}{(n+2)t-2n}>n$.
But, $\|\phi\|_t \leq C \|\phi\|_X$. Indeed, $t < \frac{2n}{n-2}$ and $t\geq  \frac{2n}{n+2} \geq s$ and we conclude by interpolation. In addition, by the dominated convergence Theorem,
$$\int_{\mathbb{R}^n}|g'(U_P)-g'(U_{P_m})|^{l} \rightarrow 0$$
Indeed using that $(x+y)^l \leq 2^lmax(x,y)^l\leq 2^l(x^l+y^l)$ for all $x,y>0$, one has
$$|g'(U_P)-g'(U_{P_m})|^{l}\leq C(|g'(U_{P})|^l+|g'(U_{P_m})|^l)=C(|g'(Q(P)^aU(Q(P)^bx))|^l+|g'(Q(P_m)^aU(Q(P_m)^bx))|^l)$$
$$ \leq 2C|g'(M^aU(\nu^bx))|^l$$ 

with $M=\|Q\|_{\infty}$ because $Q(P) \in [\nu,M]$ for all $P\in E$ and $U$  is decreasing. 

$\\$
Likewise,
$$\|(Q(P)(g'(U_P)-g'(U_{P_m}))\phi\|_{\frac{ns}{n+2s}} \leq C\|g'(U_P)-g'(U_{P_m})\|_\frac{n}{2}\|\phi\|_X\quad \text{and} \quad \|g'(U_P)-g'(U_{P_m})\|_\frac{n}{2} \rightarrow 0$$

 \noindent  Conclusion, 
$\|i^*((Q(P)(g'(U_P)-g'(U_{P_m}).)\|_X \rightarrow 0$.

 \noindent  By the same arguments, we deduce
$$\|i^*((f'(U_{P})-f'(U_{P_m})).)\|_{L(X)}\rightarrow 0.$$
Using the fact that $Q$ is continuous and bounded, one has,

$$\|i^*((Q(P)-Q(P_m))g'(U_P).)\|_{X} \rightarrow 0.$$

 \noindent  Conclusion,
$$\|L_{P} -L_{P_m}\| \rightarrow 0$$

$\\$
\end{proof}
Now introduce $\Pi_{\epsilon,P}:X \rightarrow X$ the orthogonal projection on $K_{\epsilon,P}^\perp$ (for the scalar product of $D_{1,2}(\mathbb{R}^n)$).
Now we want to show that   $\tilde{L}_{\epsilon,P}:=\Pi_{\epsilon,P}L_{\epsilon,P}:K_{\epsilon,P}^\perp \rightarrow K_{\epsilon,P}^\perp$ is invertible, its inverse is continuous and, most importantly, the norm of its inverse can be controlled by a constant independent of $P$ and $\epsilon$. To achieve this goal, we need 

\medskip
\begin{lemma}\label{Projection prop}

Consider two compact sets $E_1$ and $E_2$  of $~\mathbb{R}^{n}$ such that $d(E_1,E_2)>0$ and set $E:=E_1\times E_2$. 
Denote $\tau_{\epsilon,P}:\phi \in X \mapsto \phi(. +\frac{P}{\epsilon}) $. One has
\begin{itemize}
    \item There exists $\epsilon_0 >0$, There exists $C>0$ such that $\forall 0<\epsilon <\epsilon_0,\forall P\in E, ||\Pi_{\epsilon,P}||_X\leq C$. 
    
    \item There exists $\epsilon_0$ such that for all $\epsilon <\epsilon_0$, $P \mapsto \Pi_{\epsilon,P}$ is continuous.
    \item Consider sequences $(\epsilon_m)\in (\mathbb{R}_+^*)^\mathbb{N}$ and $P_m \in E^\mathbb{N}$ such that $\epsilon_m \rightarrow 0$ and $P_m \rightarrow (A_1,A_2)$. Then, $\tau_{\epsilon_m,P_{m,k}}\Pi_{\epsilon_m,P_m}\tau_{\epsilon_m,P_{m,k}}^{-1} \eta \rightarrow \Pi_{\infty,A_k}\eta$ in $X$ for all  $\eta \in C_c^\infty(\mathbb{R}^n)$ where $\Pi_{\infty,A_k}$ is the orthogonal projection on $ K_{\infty,A_k}^{\perp}$ where  $K_{\infty,A_k}=Span\{\frac{\partial U_{A_k}}{\partial x_j},j\in\{1,...,n\}\}$.
    \end{itemize}

\end{lemma}

\begin{proof}\label{projection}
Consider $P \in E_1 \times E_2$ and $\epsilon >0$. Recall that $K_{\epsilon,P}=Span\{\frac{\partial U_{\epsilon,k}}{\partial x_j},k\in \{1,2\},j\in\{1,...,n\}\}$. Firstly, the family $(\frac{\partial U_{\epsilon,k}}{\partial x_j})$ is free. Indeed, $(\frac{\partial U_{\epsilon,k}}{\partial x_j})_{j\in \{1,...,n\}}  $ is orthogonal for all $k$ because $U$ is radial and we conclude using translations and the fact that $\frac{P_1-P_2}{\epsilon}\rightarrow +\infty$ because $d(E_1,E_2)>0$. 
$\\$ Then if we call $G_{\epsilon,P}$ the Gram matrix of this family, we have for $\phi \in X$
$$\Pi_{\epsilon,P}\phi = \phi -\sum_{k,j} c_{\epsilon,P,k,j}\frac{\partial U_{\epsilon,k}}{\partial x_j}.$$

with $c_{\epsilon,P}=G_{\epsilon,P}^{-1}(\langle \frac{\partial U_{\epsilon,k}}{\partial x_j}|\phi \rangle)_{k,j}$. We conclude easily to obtain first two points ( to control the Gram matrix uniformly in $P$, we use that $G_{\epsilon,P}^{-1}=(D_{\epsilon,P}+G_{\epsilon,P}-D_{\epsilon,P})^{-1}$ with $D_{\epsilon,P}$ the diagonal of $G_{\epsilon,P}$). The last point is a consequence of the fact that $$\tau_{\epsilon_m,P_{m,k_0}}\Pi_{\epsilon_m,P_m}\tau_{\epsilon_m,P_{m,k_0}}^{-1}=\phi -\sum_{k,j} \tilde{c}_{\epsilon,P,k,j}\frac{\partial U_{P_{m,k}}}{\partial x_j}\left(x-\frac{P_{m,k}-P_{m,k_0}}{\epsilon_m}\right)$$
and the obvious Gram matrix. 
\end{proof}
\begin{proposition}\label{Uniformity}

Consider two compact sets $E_1$ and $E_2$  of $~\mathbb{R}^{n}$ such that $d(E_1,E_2)>0$ and set $E:=E_1\times E_2$. Then, there exists $\epsilon_0$ such that  there exists $C>0$ such that

$$\forall P\in E, 0<\epsilon <\epsilon_0,\forall\phi \in K_{\epsilon,P}^\perp, ||\tilde{L}_{\epsilon,P} \phi||_X \geq C||\phi||_X$$

$$\\$$
\end{proposition}

\begin{proof}

$\\$$\\$

Assume by contradiction the existence of sequences $(\epsilon_m) \in (\mathbb{R}_+^*)^{\mathbb{N}}, (P_m) \in E^{\mathbb{N}}, (\phi_m) \in (K_{\epsilon_m,P_m}^\perp)^{\mathbb{N}}$ such that $\epsilon_m \rightarrow 0$, $||\phi_m||_X=1$ and $||\tilde{L}_{\epsilon_m,P_m}\phi_m||_X \leq \frac{1}{m}$. By compactness, we can suppose that $P_{m,k} \rightarrow A_k$ for $k\in \{1,2\}$. Fix $k_0\in \{1,2\}$. 
$\\$
Let $\tilde{\phi}_{m,k_0} (x)=\phi_m(x+\frac{P_{m,k_0} }{\epsilon_m})$. Thus

$$||\tilde{\phi}_{m,k_0} ||_X= ||\phi_{m}||_X=1.$$

Hence, by reflexivity of $X$, there exists $\psi_{k_0} \in X$ such that $\tilde{\phi}_{m,{k_0}} \rightharpoonup \psi_{k_0} \in X$. We want to show that

$$\psi_{k_0} = i^*((f'(U_{A_{k_0}})-Q(A_{k_0})g'(U_{A_{k_0}}))\psi_{k_0}). $$
Fix $\eta \in C_c^{\infty}(\mathbb{R}^n)$. Denote $\tau$ the isometry 
\begin{equation*}
\begin{cases}

X \rightarrow X \\
 \phi \mapsto \phi(.+\frac{P_{m,{k_0}}}{\epsilon_m})
\end{cases}.
\end{equation*}
We claim that
because $\tau$ is an isometry
$$\langle\tau\tilde{L}_{\epsilon_m,P_m}\phi_m,\eta\rangle=o(1).$$
Thus, because $ \Pi_{\epsilon_m,P_m}\phi_{m}=\phi_m$,
$$\left\langle\left(\tilde{\phi}_{m,{k_0}} -\tau\Pi_{\epsilon_m,P_m}\sum i^*((f'(\ U_{\epsilon,P_{m,k}})-Q(P_{m,k})g'( U_{\epsilon,P_{m,k}}){\phi}_{m})\right)|\eta\right\rangle  =o(1).$$
But, $\langle \tilde{\phi}_{m,{k_0}}|\eta\rangle\rightarrow \langle\psi_{k_0}|\eta\rangle$ by weak convergence. In addition,

$$\left\langle\tau\Pi_{\epsilon_m,P_m}\sum i^*((f'(\ U_{\epsilon,P_{m,k}})-Q(P_{m,k})g'( U_{\epsilon,P_{m,k}}){\phi}_{m})|\eta\right\rangle$$$$=\left\langle\sum i^*((f'(\ U_{\epsilon,P_{m,k}})-Q(P_{m,k})g'( U_{\epsilon,P_{m,k}}){\phi}_{m})|\Pi_{\epsilon_m,P_m}\tau^{-1}\eta\right\rangle$$
$$=\left\langle\sum (f'(\ U_{\epsilon,P_{m,k}})-Q(P_{m,k})g'( U_{\epsilon,P_{m,k}}){\phi}_{m})|\Pi_{\epsilon_m,P_m}\tau^{-1}\eta\right\rangle_2$$
( by definition of $i^*$)
$$=\left\langle \sum \left(f'( U_{P_{m,k}})\left(x-\frac{P_{m,k}-P_{m,{k_0}}}{\epsilon_m}\right)-Q(P_{m,k})g'( U_{P_{m,k}})\left(x-\frac{P_{m,k}-P_{m,{k_0}}}{\epsilon_m}\right)\right)\tilde{\phi}_{m,{k_0}}|\tau\Pi_{\epsilon_m,P_m}\tau^{-1}\eta\ \right\rangle_2$$

$\\$

Furthermore,
$$||f'( U_{P_{m,{k_0}}})-Q(P_{m,{k_0}})g'( U_{P_{m,{k_0}}})\tilde{\phi}_{m,{k_0}}
-(f'(U_{A_{k_0}})-Q(A_{k_0})g'(U_{A_{k_0}}))\psi_{k_0}||_\frac{2n}{n+2} \rightarrow 0.$$

Indeed,
$$||(f'( U_{P_{m,{k_0}}})-Q( P_{m,{k_0}})g'( U_{P_{m,{k_0}}}))\tilde{\phi}_{m,{k_0}}
-(f'(U_{A_{k_0}})-Q(A_{k_0})g'(U_{A_{k_0}}))\psi_{k_0}||_\frac{2n}{n+2} $$
$$
\leq ||(f'(\ U_{P_{m,{k_0}}})-f'(U_{A_{k_0}})-[Q(P_{m,{k_0}})\ g'( U_{P_{m,{k_0}}})-Q(A_{k_0})g'(U_{A_{k_0}})])\tilde{\phi}_{m,{k_0}}
||_\frac{2n}{n+2} $$ 

 $$+||(f'(U_{A_{k_0}})-Q(A_{k_0})g'(U_{A_{k_0}}))(\tilde{\phi}_{m,{k_0}}-\psi_{k_0})||_{\frac{2n}{n+2}}.$$

But $h \in X \mapsto (f'(U_{A_{k_0}})-Q(A_{k_0})g'(U_{A_{k_0}}))h) \in L^{\frac{2n}{n+2}}$ is a compact operator. Thus, $$||(f'(U_{A_{k_0}})-Q(A_{k_0})g'(U_{A_{k_0}}))(\tilde{\phi}_{m,{k_0}}-\psi_{k_0})||_{\frac{2n}{n+2}} \rightarrow 0.$$

Now, let $D_m =f'(\ U_{P_{m,{k_0}}})-f'(U_{A_{k_0}})-[Q(P_{m,{k_0}})\ g'( U_{P_{m,{k_0}}})-Q(A_{k_0})g'(U_{A_{k_0}})])$.
 One has
$$||D_m \tilde{\phi}_{m,{k_0}}||_{ \frac{2n}{n+2}} \rightarrow 0$$
 Indeed,
by the inequality of Hölder with $\frac{1}{\frac{2n}{n+2}}=\frac{1}{2}+\frac{1}{n}$,

$$||D_m\tilde{\phi}_{m,{k_0}}||_{\frac{2n}{n+2}} \leq C||\tilde{\phi}_{m,{k_0}}||_X ||D_{m}||_n $$
But  $||\tilde{\phi}_{m,{k_0}}||=1$. Also, recall that $U_P(x)=Q(P)^aU(Q(P)^bx)$,$Q(P) \in [\nu,M]$ and $U$ is decreasing. Thus $|D_m|\leq C[|U(\nu^b.)^{p-1}|+|U(\nu^b.)^{q-1}| $ and $$|||U(\nu^b.)^{p-1}|+|U(\nu^b.)^{q-1}|~||_n\leq C'||U^{q-1}||_n<+\infty.$$ In addition, by continuity of $Q$ and $U$, $D_m \rightarrow 0$ almost everywhere. We conclude by the dominated convergence theorem that $$||D_{m}||_n\rightarrow 0.$$ Then, 
$$||D_m \tilde{\phi}_{m,{k_0}}||_{ \frac{2n}{n+2}} \rightarrow 0.$$

$\\$

In addition thanks to the lemma \ref{Projection prop}, $$||\tau\Pi_{\epsilon_m,P_m}\tau^{-1}\eta -\Pi_{\infty,A_{k_0}}\eta||_{2^*}\rightarrow0.$$

Thus, by the inequality of Holder, 
$$\left\langle \tau \Pi_{\epsilon_m,P_m}i^*((f'(\ U_{\epsilon,P_{m,{k_0}}})-Q(P_{m,{k_0}})g'( U_{\epsilon,P_{m,{k_0}}}){\phi}_{m}))|\eta\right\rangle$$$$\rightarrow \langle\Pi_{\infty,A_{k_0}} i^*((f'(U_{A_{k_0}})-Q(A_{k_0})g'(U_{A_{k_0}}))\psi_{k_0})|\eta \rangle $$
Finally, when $k\ne k_0$,
$$\left|\left\langle  \left(f'( U_{P_{m,k}})\left(x-\frac{P_{m,k}-P_{m,{k_0}}}{\epsilon_m}\right)-Q(P_{m,k})g'( U_{P_{m,k}})\left(x-\frac{P_{m,k}-P_{m,{k_0}}}{\epsilon_m}\right)\right)\tilde{\phi}_{m,{k_0}}|\tau\Pi_{\epsilon_m,P_m}\tau^{-1}\eta\right\rangle_2\right|$$
$$=\left|\left\langle  \left(f'( U_{P_{m,{k_0}}})\left(x-\frac{P_{m,k}-P_{m,{k_0}}}{\epsilon_m}\right)-Q(P_{m,{k}})g'( U_{P_{m,k}})\left(x-\frac{P_{m,k}-P_{m,{k_0}}}{\epsilon_m}\right)\right)\tilde{\phi}_{m,{k_0}}|\Pi_{\infty,A_{k_0}}\eta\right\rangle_2\right|+o(1)$$
$$\leq ||\tilde{\phi}_{m,{k_0}}||_2\left|\left| \left(f'( U_{P_{m,k}})\left(x-\frac{P_{m,k}-P_{m,{k_0}}}{\epsilon_m}\right)-Q(P_{m,k})g'( U_{P_{m,k}})\left(x-\frac{P_{m,k}-P_{m,{k_0}}}{\epsilon_m}\right)\right)\tilde{\phi}_{m,{k_0}})\Pi_{\infty,A_{k_0}}\eta\right|\right|_2 $$$$+o(1) $$
But, $||\Pi_{\infty,A_{k_0}}\eta||_2<+\infty $. Then, by the dominated convergence Theorem,
$$\left|\left| \left(f'( U_{P_{m,{k}}})\left(x-\frac{P_{m,k}-P_{m,{k_0}}}{\epsilon_m}\right)-Q(P_{m,k})g'( U_{P_{m,k}})\left(x-\frac{P_{m,k}-P_{m,{k_0}}}{\epsilon_m}\right)\right)\Pi_{\infty,A_{k_0}}\eta\right|\right|_2 \rightarrow 0.$$

$$\left|\left\langle  \left(f'( U_{P_{m,k}})\left(x-\frac{P_{m,k}-P_{m,{k_0}}}{\epsilon_m}\right)-Q(P_{m,{k_0}})g'( U_{P_{m,k}})\left(x-\frac{P_{m,{k}}-P_{m,{k_0}}}{\epsilon_m}\right)\right)\tilde{\phi}_{m,{k_0}}|\tau\Pi_{\epsilon_m,P_m}\tau^{-1}\eta\right\rangle_2\right|\rightarrow0$$

Conclusion, 
$$\left\langle\left(\tilde{\phi}_{m,{k_0}} -\tau\Pi_{\epsilon_m,P_m}\sum i^*((f'(\ U_{\epsilon,P_{m,k}})-Q(P_{m,k})g'( U_{\epsilon,P_{m,k}}){\phi}_{m})\right)|\eta\right\rangle \rightarrow $$$$\langle \psi_{k_0} -\Pi_{\infty,A_{k_0}} i^*((f'(U_{A_{k_0}})-Q(A_{k_0})g'(U_{A_{k_0}}))\psi_{k_0})|\eta\rangle$$
In conclusion, by uniqueness of the limit and because it is true for all $\eta$,
$$\psi_{k_0} = \Pi_{\infty,A_{k_0}}i^*((f'(U_{A_{k_0}})-Q(A_{k_0})g'(U_{A_{k_0}}))\psi_{k_0}) $$
Furthermore, since $\phi_m \in K_{\epsilon_m,P_m}^\perp$
$$\left\langle \phi_{m}\big| \frac{\partial U_{\epsilon_m,P_{m,k_0}}}{\partial x_j}\right\rangle=0$$
$$\left\langle \tilde{\phi}_{m,{k_0}}\big| \frac{\partial U_{P_{m,k_0}}}{\partial x_j}\right\rangle=0$$
But 

$$\left|\left|\frac{\partial U_{P_{m,{k_0}}}}{\partial x_j}-\frac{\partial U_{A_{k_0}}}{\partial x_j}\right|\right|_{D^{1,2}(\mathbb{R}^n)}\rightarrow 0$$ 
and 
$$\left\langle \tilde{\phi}_{m,{k_0}}| \frac{\partial U_{A_{k_0}}}{\partial x_j}\right\rangle \rightarrow \left\langle \psi_{k_0}|\frac{\partial U_{A_{k_0}}}{\partial x_j}\right  \rangle$$
Thus, taking the limit

$$\psi_{k_0} \in K_{\infty,A_{k_0}}^\perp$$

where $K_{\infty,A_{k_0}}^\perp=\{v \in X, \langle v|\frac{\partial U_{A_{k_0}}}{\partial x_j}\rangle=0\}$. Set, $L_{\infty,A_{k_0}}(\phi)=\phi - i^*((f'(U_{A_{k_0}})-Q(A_{k_0})g'(U_{A_{k_0}}))\phi)$.
Then,

$$\Pi_{\infty,A_{k_0}}L_{\infty,A_{k_0}}\psi_{k_0} =0$$
But, thanks to lemma \ref{non-degeneracy of the solution}, $$\left\langle L_{\infty,A_{k_0}}\psi_{k_0}|\frac{\partial U_{A_{k_0}}}{\partial x_j}\right\rangle = \left\langle \psi_{k_0}|L_{\infty,A_{k_0}}\frac{\partial U_{A_{k_0}}}{\partial x_j}\right\rangle=0 $$
Then, $$L_{\infty,A_{k_0}}\psi_{k_0} =0$$
thanks to the lemma \ref{non-degeneracy of the solution}, $\psi_{k_0}=0$.

Now,
$$ \frac{1}{m} \geq ||\tilde{L}_{\epsilon_m,P_m} \phi_m||_X \geq ||\phi_m||- ||\Pi_{\epsilon_m,P_m}\sum i^*((f'(\ U_{\epsilon_m,P_{m,k}})-Q(P_{m,k})g'( U_{\epsilon_m,P_{m,k}}){\phi}_{m})||_X $$ $$= 1 - ||\Pi_{\epsilon_m,P_m}\sum i^*((f'(\ U_{\epsilon_m,P_{m,k}})-Q(P_{m,k})g'( U_{\epsilon,P_{m,k}}))\phi_{m})||_X.$$
 Thanks to lemma \ref{Projection prop}, it suffices to show that $||\sum i^*((f'(\ U_{\epsilon,P_{m,k}})-Q(P_{m,k})g'( U_{\epsilon,P_{m,k}}))\phi_{m})||_X \rightarrow 0$ to conclude to a contradiction.

But with $\beta=\frac{ns}{n+2s}$ or $\beta=\frac{2n}{n+2}$ and $k_0 \in \{1,2\}$
$$ ||(f'(\ U_{\epsilon_m,P_{m,{k_0}}})-Q(P_{m,{k_0}})g'( U_{\epsilon_m,P_{m,{k_0}}}))\phi_{m})||_\beta$$$$=||(f'(\ U_{P_{m,{k_0}}})-Q(P_{m,{k_0}})g'( U_{P_{m,{k_0}}}))\tilde{\phi}_{m,{k_0}})||_\beta=
||(f'( U_{A_{k_0}})-Q(A_{k_0})g'( U_{A_{k_0}}))\tilde{\phi}_{m,{k_0}}||_\beta+o(1)$$

because, $P \rightarrow f'( U_{P})-Q(P)g'( U_{P})$ is continuous according to the lemma \ref{ContinuityP} and $$||\tilde{\phi}_{m,{k_0}}||_X=1.$$
But 
$h \mapsto (f'( U_{A_{k_0}})-Q(A_{k_0})g'( U_{A_{k_0}}))h $ is compact. Thus,
$$ ||(f'(\ U_{\epsilon_m,P_{m,{k_0}}})-Q(P_{m,{k_0}})g'( U_{\epsilon_m,P_{m,{k_0}}}))\phi_{m}) ||_\beta \rightarrow 0.$$

$\\$

\end{proof}

\begin{corollary}\label{inversibility}
$\\$
Consider a compact set $E_1\times E_2$ of $\mathbb{R}^{2n}$ such that $d(E_1,E_2)>0$,
$\tilde{L}_{\epsilon,P}$ is invertible for all $P \in E$ and its inverse is continuous. In addition, there exists $C$ a constant and $\epsilon_0>0$ such that $||\tilde{L}^{-1}_{\epsilon,P}||\leq C$ for all $\epsilon <\epsilon_0$ and $P \in E$. 

\end{corollary}

\begin{proof}

Thanks to the previous results
\begin{itemize}
    \item By the lemma \ref{Compact}, $L_{\epsilon,P}$ is of the form $id - K$ with $K$ a compact operator. Thus $\tilde{L}_{\epsilon,P}$ is also of this form.
    \item By the proposition \ref{Uniformity} it is clear that $\tilde{L}_{\epsilon,P}$ is one-to-one.
\end{itemize}
In conclusion, the Fredholm alternative proves that $\tilde{L}_{\epsilon,P}$ is a bijective, continous operator from $K_{\epsilon,P}^\perp$ to $K_{\epsilon,P}^\perp$. Then, thanks to the open mapping theorem because  $K_{\epsilon,P}^\perp$ is a Banach space ( since closed in $X$), $\tilde{L}_{\epsilon,P}^{-1}$ is continuous.
\smallskip

In addition, by the proposition \ref{Uniformity},
There exists $\epsilon_0$ such that for all compact set $E$ of $~\mathbb{R}^{2n}$ of the form $E_1\times E_2$ with $d(E_1,E_2)>0$, there exists $C>0$ such that

$$\forall P\in E, 0<\epsilon <\epsilon_0,\forall\phi \in K_{\epsilon,P}^\perp, ||\tilde{L}^{-1}_{\epsilon,P} \phi||_X\leq C.$$
\end{proof} 

 \noindent  One can rewrite the equation \eqref{Main-single}
$$L_{\epsilon,P}\phi = N_f(\phi) + N_g(\phi) + I_f + I_g +  I_{f'} + I_{g'}  + \sum i^*(g'(U_{\epsilon,k})\phi(Q(P_k)-Q_{\epsilon}))+  \sum i^*((Q(P_k)-Q_{\epsilon}(x))g(U_{\epsilon,k}))$$
thus $$\tilde{L}_{\epsilon,P}\phi = \Pi_{\epsilon,P}(N_f(\phi) + N_g(\phi) + I_f + I_g +    I_{f'} + I_{g'}  $$$$+ \sum i^*(g'(U_{\epsilon,k})\phi(Q(P_k)-Q_{\epsilon}))+  \sum i^*((Q(P_k)-Q_{\epsilon}(x))g(U_{\epsilon,k})))$$
i.e. 
$$\phi = \tilde{L}_{\epsilon,P}^{-1}\Pi_{\epsilon,P}(N_f(\phi) + N_g(\phi) + I_f + I_g   +  I_{f'} + I_{g'}  + \sum i^*(g'(U_{\epsilon,k})\phi(Q(P_k)-Q_{\epsilon}))+ $$$$ \sum i^*((Q(P_k)-Q_{\epsilon}(x))g(U_{\epsilon,k})) $$
i.e. $$\phi = T_{\epsilon,P}\phi $$

\medskip

$\\$

Thus the last equation takes the form 
$\phi = T_{\epsilon,P}(\phi)$ for some $\phi\in K_{\epsilon,P}^\perp$. The definition of $T_{\epsilon,P}$ show it is enough to find a fixed point in $X$.  
To apply the Banach fixed point theorem, one needs some controls on the terms which play a role in this equation. That is the goal of the next section. 
\subsubsection{Control of terms}
The first result that we prove is the following one:
\begin{proposition}\label{Stability2}

 For every compact set $E$ of the form $E_1\times E_2$ with $d(E_1,E_2)>0$, there exists $\epsilon_0$ and a constant $N$ such that for all $\epsilon < \epsilon_0$ and for all $P \in E$, $T_{\epsilon,P}$ maps the ball $B(0,N\epsilon)$ into itself.
$$\\$$
\end{proposition}

\begin{proof}
    
Firstly thanks to Proposition \ref{Uniformity}, we can find $C>0$ and $\epsilon_{0}$ such that $$\forall \epsilon <\epsilon_0,\|\tilde{L}_{\epsilon,P}^{-1}\Pi_{\epsilon,P}\|\leq C. $$
Thus, it suffices to control 
$$N_f(\phi) + N_g(\phi) + I_f + I_g   +  I_{f'} + I_{g'}  + \sum i^*(g'(U_{\epsilon,k})\phi(Q(P_k)-Q_{\epsilon}))+  \sum i^*((Q(P_k)-Q_{\epsilon}(x))g(U_{\epsilon,k})$$
\underline{Control of 
$\\$
$N_f(\phi) +N_g(\phi)$:}
$\\$$\\$
As in \cite{MichelettiPistoia2003}, using the following true inequalities (true because $n\geq9$ then $q<\frac{n+2}{n-2}<2$),
$$ \left\{
    \begin{array}{ll}
        \forall x_1,x_2\in \mathbb{R}, |g(x_1)-g(x_2)-g'(x_1)(x_1-x_2)|\leq C|x_1-x_2|^q
        \\
        ~~~~~\forall x_1,x_2\in \mathbb{R},|g'(x_1)-g'(x_2)|\leq C|x_1-x_2|^{q-1}
    \end{array}
\right.$$

we have
$$ \left\{
    \begin{array}{ll}
        |g(S_{\epsilon,P} + \phi_1) - g(S_{\epsilon,P} + \phi_2)- g'(S_{\epsilon,P}+\phi_2)(\phi_1-\phi_2)| \leq C|\phi_1-\phi_2|^q
        \\
        ~~~~~ |g'(S_{\epsilon,P} +\phi_1)-g'(S_{\epsilon,P}+\phi_2)| \leq C|\phi_1-\phi_2|^{q-1}
    \end{array}
\right.$$
$$$$

 \noindent  The same inequalities are true for $f$ mutatis mutandis.

$\\$$\\$
We have thus,

$$\|N_g(\phi)\|_X\leq C\|\phi\|_X^q$$
because
\begin{itemize}
    
    \item $\|\phi^q\|_{\frac{2n}{n+2}} \leq \|\phi\|_X^q$ because $s\leq \frac{2n}{n+2}q \leq 2^*$
    \item $\|\phi^q\|_{\frac{ns}{n+2s}} \leq \|\phi\|_X^q$ because $s\leq \frac{ns}{n+2s}q \leq 2^*$ (because $s \leq \frac{n}{2}(q-1)$ and $s\leq 2^*$ thus $\frac{ns}{n+2s}q\leq \frac{ns}{n+2s}\frac{n+2}{n-2} \leq 2^*$).
\end{itemize}
Likewise
$$\|N_f(\phi)\|_X\leq C\|\phi\|_X^p.$$

\noindent Thus, because $q<p$, we have when $\|\phi\|_X\leq 1$,
$$\|N_g(\phi)+N_f(\phi)\|_X\leq C\|\phi\|_X^q.$$

$\\$
\underline{Control of 
$\sum i^*((Q(P_k)-Q_{\epsilon} )g(U_{\epsilon,k}))$:}
$\\$$\\$
We have

$$\|\sum i^*((Q(P_k)-Q_{\epsilon}(x))g(U_{\epsilon,P_k}))\|_X\leq\sum \|i^*((Q(P_k)-Q_{\epsilon}(x))g(U_{\epsilon,P_k}))\|_X.$$

$\\$

 \noindent  Consider $k\in \{1,2\}$.

$\\$$\\$
Let $\beta=\frac{2n}{n+2}$ or $\beta = \frac{ns}{n+2s}$. One has

    $$\|(Q(P_k)-Q_{\epsilon}(x))g(U_{\epsilon,P_k})\|_{\beta}^{\beta}=\|(Q(\epsilon x+ P_k)-Q(P_k))U_{P_k}^q\|_{\beta}^{\beta}$$$$\leq \int_{\mathbb{R}^n}|Q(\epsilon x+P_k)-Q(P_k)|^{\beta}U_{P_k}^{\beta q}dx \leq C\|\nabla Q\|_\infty^\beta\epsilon^{\beta}\int_{\mathbb{R}^n}|x|^{\beta}U^{q\beta}dx.$$
But by hypothesis, $\beta(\alpha q-1)>n$ (because $n>4$ or $\frac{n}{(n-2)q-3}<s$ according to the values of $\beta$).

$\\$
We conclude that
$$\|(Q(\epsilon x+ P_k)-Q(P_k))U_{P_k}^q\|_{\beta} \leq C\epsilon. $$

\noindent Conclusion:

$$\|\sum i^*((Q(P_k)-Q_{\epsilon}(x))g(U_{\epsilon,P_k}))\|_X\leq C\epsilon.$$

$\\$$\\$
\underline{Control of 
$\sum i^*((Q_{\epsilon} - Q(P_k))g'(U_{\epsilon,k})\phi)$:}

$\\$$\\$
Consider $k\in \{1,2\}$. Consider $t$ such that $ \frac{2n}{n+2}<t < 2$. Reasoning as before with Hölder's inequality we obtain 
$$\|U_{P_k}^{q-1}(Q(P_k)-Q(\epsilon x+P_k))\phi\|_{\frac{2n}{n+2}}\leq C\|\phi\|_t \epsilon\|U_{P_k}^{q-1}|x\||_{\frac{t2n}{(n+2)t-2n}}$$ 

because $ |Q({P_k}) - Q(\epsilon x+ {P_k})|\leq \epsilon M|x|$. 
But 
$[\alpha(q-1)-1]t\frac{2n}{t(n+2)-2n} >n$ (because $t < 2$).
Then,
$$\|U_{P_k}^{q-1}(Q(P_k)-Q(\epsilon x+P_k))\phi\|_{\frac{2n}{n+2}}\leq C\|\phi\|_X \epsilon.$$

\noindent In addition,

$$\|U_{P_k}^{q-1}(Q(P_k)-Q(\epsilon x+P_k))\phi\|_{\frac{sn}{n+2s}} \leq \|U_{P_k}^{q-1}sup_{P\in E}|Q(P_k)-Q(\epsilon x+P_k)|~~\|_{\frac{n}{2}}\|\phi\|_X =o(1)\|\phi\|_X$$ by the dominated convergence Theorem because $Q$ bounded and using that $sup|Q(P_k)-Q(\epsilon x+P_k))|\leq  \epsilon M|x|$ to show the pointwise convergence. 
$\\$
Conclusion (uniformly in $P$)

$$\|\sum i^*((Q_{\epsilon} - Q(P_k))g'(U_{\epsilon,k})\phi)\|_X \leq o(1)\|\phi\|_X.$$

$\\$$\\$
\underline{Control of 
$I_f+I_g$:}
$\\$

\noindent For the control, we need the following lemma
$\\$$\\$
\begin{lemma}\label{Control puissance q}
$\\$

For all $a>0$, $b$ real and $e> 0$ we have:

\begin{itemize}
    \item  if $e<1$:
    $$\|a+b|^{e}-a^{e}|\leq C(e)min\{|b|^{e},a^{e-1}|b|\}$$
    \item if $e\geq 1$:
    $$\|a+b|^{e}-a^{e}|\leq C(e)(|b|^{e}+a^{e-1}|b|)$$

\end{itemize}

\end{lemma}
$\\$
$\\$$\\$

By hypothesis, $\rho=d(E_1,E_2)>0$. Set $r=\frac{\rho}{2\epsilon}$
Let $\beta = \frac{2n}{n+2}$ or $\beta  =\frac{ns}{n+2s}$. Then, 
$$\|(U_{\epsilon,1}+U_{\epsilon,2})^q-(U_{\epsilon,1}^q+U_{\epsilon,2}^q)\|^{\beta}_{\beta} $$$$ =
\int_{\mathbb{R}^n}\left|\left[U_{P_1}\left(x-\frac{P_1}{\epsilon}\right)+U_{P_2}\left(x-\frac{P_2}{\epsilon}\right)\right]^q-\left[U^q_{P_1}\left(x-\frac{P_1}{\epsilon}\right)+U_{P_2}^q\left(x-\frac{P_2}{\epsilon}\right)\right]\right|^\beta dx$$
$$=\int_{\mathbb{R}^n}|U_{P_1}(x)+U_{P_2}(x-\frac{P_2-P_1}{\epsilon}))^q-(U^q_{P_1}(x)+U_{P_2}^q(x-\frac{P_2-P_1}{\epsilon})|^\beta dx$$

$$ \leq 
\int_{|x|\leq r} \left|\left[U_{P_1}(x)+U_{P_2}\left(x-\frac{P_2-P_1}{\epsilon}\right)\right]^q-\left[U^q_{P_1}(x)+U_{P_2}^q\left(x-\frac{P_2-P_1}{\epsilon}\right)\right]\right|^\beta dx+$$ $$\int_{|x|\geq r} \left|\left[U_{P_1}(x)+U_{P_2}\left(x-\frac{P_2-P_1}{\epsilon}\right)\right]^q-\left[U^q_{P_1}(x)+U_{P_2}^q\left(x-\frac{P_2-P_1}{\epsilon}\right)\right]\right|^\beta dx$$
$$=: Res_1 + Res_2$$
But

$$|Res_2| \leq C\left(\int_{|x |\geq r}\left|\left[U_{P_1}(x)+U_{P_2}\left(x-\frac{P_2-P_1}{\epsilon}\right)\right]^q-U_{P_2}\left(x-\frac{P_2-P_1}{\epsilon}\right)\right|^\beta  + \int_{|x|\geq r}|U_{P_1}^q(x)|^\beta \right).$$

Thus, by Lemma \ref{Control puissance q} (case $e \geq 1$)
$$|Res_2| \leq C\left(\int_{|x|\geq r}\left|U_{P_2}^{q-1}\left(x-\frac{P_2-P_1}{\epsilon}\right)U_{P_1}(x)\right|^\beta + \int_{|x|\geq r}U_{P_1}^{\beta q}(x)\right).$$

But by hypothesis $\beta(q\alpha-1)>n$.

Thus,
$$\left(\int_{|x|\geq r}U_{P_1}^{\beta q}(x)\right)^{\frac{1}{\beta}} =o(\epsilon).$$

Furthermore,
take $l_{lim}=\frac{n}{(q-1)\alpha\beta}$.
We have $l_{lim}^*=\frac{n}{n-(q-1)\alpha \beta}$ such that $1= \frac{1}{l^*_{lim}}+\frac{1}{l_{lim}}$ with $n-(q-1)\alpha \beta>0$ and we have

$$l_{lim}^*\beta \alpha - n>\beta l_{lim}^*$$
Indeed this inequality is equivalent to

$$\beta(\alpha q-1) > n.$$
Now we need to apply Hölder's inequality with $l_{lim}$ to control 
$$\int_{|x|\geq r}\left|U_{P_2}^{q-1}\left(x-\frac{P_2-P_1}{\epsilon}\right)U_{P_1}(x)\right|^\beta.$$

 \noindent  Nonetheless, $$\|U^{(q-1)\beta}\|_{l_{lim}}=+ \infty$$
because
$$l_{lim} \times (q-1)\beta \alpha =n.$$

 \noindent  Nevertheless by continuity, one can find $l$ such that

$$l^*\beta\alpha-n > \beta l^*$$
and 
$$\|U^{(q-1)\beta}\|_{l}<+ \infty.$$

 \noindent  Thus by Hölder's inequality
$$\int_{|x|\geq r}\left|U_{P_2}^{q-1}\left(x-\frac{P_2-P_1}{\epsilon}\right)U_{P_1}(x)\right|^\beta \leq C\|U^{(q-1)\beta}\|_l\|U_{P_1}^\beta\|_{L^{l^*}(B(0,r)^c)}.$$

 \noindent  Thus 
$$\left\|\int_{|x - \frac{P_2-P_1}{\epsilon}|\geq r}\left|U_{P_1}^{q-1}(x)U_{P_2}\left(x-\frac{P_2-P_1}{\epsilon}\right)\right| \right\|_{\beta} \leq C\|U^\beta\|_{L^{l^*}(B(0,r)^c)}^{\frac{1}{\beta}}= o(\epsilon).$$

 \noindent  because $\frac{\alpha\beta l^*-n}{\beta l^*}>1.$
$\\$

 \noindent  In conclusion,
$|Res_2| =o(\epsilon)$
$$\\$$
In addition, if $|x-\frac{P_2-P_1}{\epsilon}| \leq r $ then $|x| \geq |\frac{P_2-P_1}{\epsilon}|-|x-\frac{P_2-P_1}{\epsilon}|\geq \frac{\rho}{\epsilon}-\frac{\rho}{2\epsilon}=\frac{\rho}{2\epsilon}$

 \noindent  Thus, by same arguments, we have
$$|Res_1| =o(\epsilon).$$

$$$$
Hence

$$\|I_g\|_X =o(\epsilon). $$
Likewise
$$\|I_f\|_X =o(\epsilon). $$
$\\$$\\$
  \noindent  \underline{Control of 
$I_{f'}+I_{g'}$:}
$\\$
$\\$
Take $\beta = \frac{n}{2}$ or $\beta = \frac{t2n}{(n+2)t-2n} $
$$
||f'(S_{\epsilon,P}) - \sum f'(U_{\epsilon,k})||^{\beta}_{\beta}=||(U_{\epsilon,1}+U_{\epsilon,2})^{p-1}-(U_{\epsilon,1}^{p-1 }+U_{\epsilon,2}^{p-1})||^{\beta}_{\beta} $$
$$=\int_{\mathbb{R}^n}\left|\left[U_{P_1}(x)+U_{P_2}\left(x-\frac{P_2-P_1}{\epsilon}\right)\right]^{p-1}-\left[U^{p-1}_{P_1}(x)+U_{P_2}^{p-1}\left(x-\frac{P_2-P_1}{\epsilon}\right)\right]\right|^\beta dx$$
$$=Res_1 + Res_2$$
$$|Res_2| \leq C\left(\int_{|x|\geq r}\left|\left[ U_{P_1}(x)+U_{P_2}\left(x-\frac{P_2-P_1}{\epsilon}\right)\right]^{p-1}-U^{p-1}_{P_2}\left(x-\frac{P_2-P_1}{\epsilon}\right)\right|^\beta  + \int_{|x|\geq r}|U_{P_1}^{p-1}(x)|^\beta \right) $$

$$ \leq C(\int_{|x|\geq r}|U_{P_1}(x)|^{\beta(p-1)} + \int_{|x|\geq r}U_{P_1}^{\beta (p-1)}(x)) $$
$$\leq C '(\int_{|x |\geq r}U^{\beta(p-1)}_{P_1}(x) ) \leq\frac{C''}{r^{\alpha\beta(p-1)-n}}=o(1)$$
Likewise
$$Res_1 = o(1).$$

 \noindent  We deduce that
$$I_{f'}= o(1)||\phi||_X.$$

 \noindent  Likewise
$$I_{g'}=o(1) ||\phi||_X.$$
$\\$$\\$

$\\$$\\$
In conclusion,

$\forall P\in E,\forall \phi \in X,||T_{\epsilon,P}(\phi)||_X \leq C(||\phi||^q+  o(1)||\phi||+\epsilon).$

We conclude easily because $q>1$.
\end{proof}
$\\$
\begin{proposition}\label{Lipschitz}
$$\\$$
$T_{\epsilon,P}$ is a contraction on $\bar{B}(0,N\epsilon)$ with a lipschitz constant that does not depend on $P$. 
$$\\$$
\end{proposition}

\begin{proof}
$$T_{\epsilon,P}\phi_1 - T_{\epsilon,P}\phi_2=(\tilde{L}_{\epsilon,P})^{-1}\Pi_{\epsilon,P}(N_f(\phi_1)-N_f(\phi_2)+N_g(\phi_1)-N_g(\phi_2) + \sum i^*(g'(U_{\epsilon,k})(Q(P_k)-Q_\epsilon)(\phi_1-\phi_2)) $$
$$+I_{f'}(\phi_1)-I_{f'}(\phi_2)+I_{g'}(\phi_1)-I_{g'}(\phi_2))$$

Firstly, as in \cite{MichelettiPistoia2003},
$$||N_f(\phi_2) - N_f(\phi_1)||_X =||i^*(f(S_{\epsilon,P}+\phi_2)-f(S_{\epsilon,P}+\phi_1)-f'(S_{\epsilon,P})(\phi_2-\phi_1))||_X $$ $$\leq ||i^*(f(S_{\epsilon,P}+\phi_2)-f(S_{\epsilon,P}+\phi_1)-f'(S_{\epsilon,P}+\phi_2)(\phi_2-\phi_1)||_X + ||(i^*((f'(S_{\epsilon,P}+\phi_2)-f'(S_{\epsilon,P}))(\phi_2-\phi_1))||_X$$
$$\leq C(||\phi_1-\phi_2||_X^p+ ||\phi_2||_X^{p-1}||\phi_1-\phi_2||_X)$$

Indeed, we use the following facts

\begin{itemize}
    \item $s\leq (q-1)\frac{n}{2}\leq (p-1)\frac{n}{2}\leq (\frac{n+2}{n-2}-1)\frac{n}{2}=2^*$ Thus, $\phi_2^{p-1}\in L^{\frac{n}{2}}; $ 
    \item In addition, $$||(i^*((f'(S_{\epsilon,P}+\phi_2)-f'(S_{\epsilon,P})(\phi_2-\phi_1))||_s\leq |||\phi_2|^{p-1}|\phi_1-\phi_2|||_{\frac{ns}{n+2s}} \leq ||\phi_2^{p-1}||_{\frac{n}{2}}||\phi_1-\phi_2||_X$$ $$=||\phi_2||_{\frac{n(p-1)}{2}}^{p-1}||\phi_1-\phi_2||_X$$

    \item Then, $$||(i^*((f'(S_{\epsilon,P}+\phi_2)-f'(S_{\epsilon,P}))(\phi_2-\phi_1))||_{\frac{2n}{n+2}}\leq ||\phi^{p-1}||_{\frac{n}{2}}||||\phi_1-\phi_2||_{\frac{2n}{n-2}}$$
    because,$\frac{1}{\frac{2n}{n+2}}=\frac{1}{\frac{n}{2}}+\frac{1}{\frac{2n}{n-2}}$.
\end{itemize}
Furthermore, 
$$||\sum i^*(g'(U_{\epsilon,k})(Q(P_k)-Q_\epsilon)(\phi_1-\phi_2))||_X \leq C||\phi_1-\phi_2||_X o(1)$$

Finally, we have by the previous estimation ( independent of $P$)
$$||(I_{f'}+I_{g'})(\phi_1)-(I_{f'}+I_{g'})(\phi_2)||\leq o(1)||\phi_1-\phi_2||_X$$
Thus, 
$||T_{\epsilon,P}(\phi_1)-T_{\epsilon,P}(\phi_2)||_X \leq  \ C(||\phi_1-\phi_2||_X^q+ ||\phi_2||^{q-1}||\phi_1-\phi_2||_X + ||\phi_1-\phi_2||_X o(1))$
\medskip
\noindent Then,
$$\| T_{\epsilon,P}(\phi_1)- T_{\epsilon,P}(\phi_2)\|_X \leq C_0(\epsilon)\|\phi_1-\phi_2\|_X $$
with $0<C_0(\epsilon) <1$ for $\epsilon $ small enough (which does not depend on $P$) because $\|\phi_1\|_X,\|\phi_2\|_X\leq N\epsilon$.

\end{proof}
We conclude by applying the following version of Banach's fixed-point theorem.

$\\$
\begin{lemma}\label{Banach}
$\\$

Consider $(Y,d)$ a complete metric space and $\Gamma \subset \mathbb{R}^m$. Consider $f: Y \times \Gamma  \rightarrow Y$. Suppose that $\gamma \mapsto f(y,\gamma)$ is continuous for all $y \in Y$ and there exists $k \in (0,1)$ such that for all $\gamma\in \Gamma$

$$d(f(x,\gamma),f(y,\gamma))\leq kd(x,y).$$
Then $f(.,\gamma)$ has a unique fixed point $x_\gamma$ for all $\gamma$. In addition, $\gamma \mapsto x_\gamma$ is continuous.   
\end{lemma}

$\\$$\\$
We need thus to show the next result to conclude
\begin{lemma}\label{continuité phi}

Consider $\phi \in X$ and two disjoints compact sets $E_1$ and $E_2$. Set $E=E_1\times E_2$. Then, $P \mapsto T_{\epsilon,P}\phi \in X$ is continuous. 

\end{lemma}

\begin{proof}
Fix a sequence $(P_m)$ such that $P_m \rightarrow P$. We have
$$||\tilde{L}_{\epsilon,P}T_{\epsilon,P}\phi -\tilde{L}_{\epsilon,P_m}T_{\epsilon,P_m}\phi||_X$$$$\geq  ||\tilde{L}_{\epsilon,P}[T_{\epsilon,P} -\Pi_{\epsilon,P}T_{\epsilon,P_m}]\phi||_X-||\tilde{L}_{\epsilon,P}\Pi_{\epsilon,P}T_{\epsilon,P_m}\phi-\tilde{L}_{\epsilon,P_m}\Pi_{\epsilon,P_m}T_{\epsilon,P_m}\phi||_X$$

because $\tilde{L}_{\epsilon,P_m}\Pi_{\epsilon,P_m}T_{\epsilon,P_m}\phi=\tilde{L}_{\epsilon,P_m}T_{\epsilon,P_m}\phi$.
Thanks to the definition $T_{\epsilon,P}$, 
$[T_{\epsilon,P}\phi -\Pi_{\epsilon,P}T_{\epsilon,P_m}]\phi \in K_{\epsilon,P}^\perp$. Then, by the proposition \ref{Uniformity}
$$||\tilde{L}_{\epsilon,P}[T_{\epsilon,P}\phi -\Pi_{\epsilon,P}T_{\epsilon,P_m}]\phi||_X\geq C||[T_{\epsilon,P}\phi -\Pi_{\epsilon,P}T_{\epsilon,P_m}]\phi||_X.$$
Thus, 

$$||T_{\epsilon,P}\phi - T_{\epsilon,P_m}\phi||\leq \frac{1}{C} [||\tilde{L}_{\epsilon,P}\Pi_{\epsilon,P}T_{\epsilon,P_m}\phi-\tilde{L}_{\epsilon,P_m}\Pi_{\epsilon,P_m}T_{\epsilon,P_m}\phi||_X+$$$$||\tilde{L}_{\epsilon,P}T_{\epsilon,P}\phi -\tilde{L}_{\epsilon,P_m}T_{\epsilon,P_m}\phi||_X]+||(Id-\Pi_{\epsilon,P})T_{\epsilon,P_m}\phi||_X$$

But,
$$||\tilde{L}_{\epsilon,P}\Pi_{\epsilon,P}T_{\epsilon,P_m}\phi-\tilde{L}_{\epsilon,P_m}\Pi_{\epsilon,P_m}T_{\epsilon,P_m}\phi||_X\leq ||\tilde{L}_{\epsilon,P}\Pi_{\epsilon,P}-\tilde{L}_{\epsilon,P_m}\Pi_{\epsilon,P_m}||_X||T_{\epsilon,P_m}\phi||_X \rightarrow 0$$
thanks to lemma \ref{ContinuityP} and \ref{Projection prop}.
But, 

$$||(Id-\Pi_{\epsilon,P})T_{\epsilon,P_m}\phi||_X\leq ||\Pi_{\epsilon,P}-\Pi_{\epsilon,P_m}||_X||T_{\epsilon,P_m}\phi||_X\rightarrow 0$$
because $T_{\epsilon,P_m}\phi=\Pi_{\epsilon,P_m}T_{\epsilon,P_m}\phi$. 
Finally,$$ \tilde{L}_{\epsilon,P}T_{\epsilon,P}\phi=\Pi_{\epsilon,P}(N_f(\phi) + N_g(\phi) + I_f + I_g   +  I_{f'} + I_{g'}  + $$$$\sum i^*(g'(U_{\epsilon,k})\phi(Q(P_k)-Q_{\epsilon}))+  \sum i^*((Q(P_k)-Q_{\epsilon}(x))g(U_{\epsilon,k})) $$

But $||\Pi_{\epsilon,P}||\leq C'$ with a constant $C'>0$ thanks to lemma \ref{Projection prop}. In addition, the same calculations than before show that
$$P \mapsto N_f(\phi) + N_g(\phi) + I_f + I_g   +  I_{f'} + I_{g'}  + $$$$\sum i^*(g'(U_{\epsilon,k})\phi(Q(P_k)-Q_{\epsilon}))+  \sum i^*((Q(P_k)-Q_{\epsilon}(x))g(U_{\epsilon,k})$$
is continuous. 
\end{proof} 
Using the lemma \ref{Banach} and \ref{continuité phi} with $Y=X\cap \bar{B}(0,N\epsilon)$ and $f= T_{\epsilon,P}$, we conclude that there exists, $\ \phi_{\epsilon,P} \in Y$ such that $T_{\epsilon,P}\phi_{\epsilon,P}=\phi_{\epsilon,P}, \forall P\in  E,\forall \epsilon<\epsilon_0 $.

\subsection{Solving the bifurcation equation}

Let $\psi_{k,j}=\frac{\partial U_{\epsilon,k  }}{\partial x_j}$. Fix $k \in \{1,2\}$ and $j \in \{1,...,n\}$. Define
$$G_{\epsilon,k,j}(P)=\langle -W_{\epsilon,P}  + i^*(f(W_{\epsilon,P} ) - g(W_{\epsilon,P} )Q_{\epsilon})|\psi_{k,j}\rangle$$ where $W= W_{\epsilon,P}=S_{\epsilon,P}+\phi_{\epsilon,P}$. It follows from the preceding result that $G_{\epsilon,k,j}$ is continuous. As in \cite{MichelettiPistoia2003}, we seek $P_{\epsilon}$  such that $G_{\epsilon,k,j}(P_{\epsilon}) =0$ for all $k$ and $j$ to conclude. We proceed as in \cite{MichelettiPistoia2003} and will show that the Brouwer degree of $P \rightarrow (\frac{G_{\epsilon,k,j}(P)}{Q(P_k)^l\epsilon})_{k,j}$ with some $l$ is not equal to 0 when $\epsilon$ is small enough. $$$$
Until now, we have worked with an arbitrary compact set $E$. Now introduce for $k\in\{1,2\}$, $R_k$ a real such that $Q_{|B(\xi_k,R_k)}$ has only $\xi_k$ as a critical point (it is possible because $\xi_k$ is a nondegenerate critical point) and such that $d(B(\xi_1,R_1),B(\xi_2,R_2))>0$. Until the end we will consider $E=\bar{B}(\xi_1,R_1)\times\bar{B}(\xi_2,R_2) $. We want to show the following result: 
\begin{proposition}\label{Brouwer 2}
$\\$$\\$
There is a constant $C$ such that
 $F_{\epsilon}:P \rightarrow (\frac{CG_{\epsilon,k,j}(P)}{Q(P_k)^{a(q+1)-bn}\epsilon})_{k,j}$ converges uniformly to $\nabla \tilde Q$ when $\epsilon \rightarrow0$ 
where $\tilde{Q}:(P_1,P_2) \rightarrow \sum_{i} Q(P_i)$.

\end{proposition}
$\\$
\begin{proof}
    
Consider ${k_0}$ and $j$. For the sake of simplicity, we will suppose $k_0=1$.

$$G_{\epsilon,{k_0},j}(P)=\langle -W_{\epsilon,P} + i^*(f(W_{\epsilon,P} ) - g(W_{\epsilon,P} )Q_\epsilon)|\psi_{{k_0},j}\rangle_{D^{1,2}(\mathbb{R}^n)} $$

$$=-\sum\langle i^*(g(U_{\epsilon,{k}})(Q_{\epsilon}-Q(P_{k}))|\psi_{{k_0},j}\rangle_{D^{1,2}(\mathbb{R}^n)} +\langle I_f|\psi_{{k_0},j}\rangle_{D^{1,2}(\mathbb{R}^n)} +\langle I_g|\psi_{{k_0},j}\rangle_{D^{1,2}(\mathbb{R}^n)} $$
$$+\langle N_f(\phi)|\psi_{{k_0},j}\rangle_{D^{1,2}(\mathbb{R}^n)} +    \langle N_g(\phi)|\psi_{{k_0},j}\rangle_{D^{1,2}(\mathbb{R}^n)} +\langle i^*((f'(S_{\epsilon,P})-f'(U_{\epsilon,{k_0}}))\phi| \psi_{{k_0},j}\rangle_{D^{1,2}(\mathbb{R}^n)} -$$$$\langle i^*((g'(S_{\epsilon,P})-g'(U_{\epsilon,{k_0}}))Q_{\epsilon}\phi|\psi_{{k_0},j}\rangle_{D^{1,2}(\mathbb{R}^n)} +\langle i^*(g'(U_{\epsilon,{k_0}})(Q(P_{k_0})-Q_{\epsilon})\phi|\psi_{{k_0},j}\rangle_{D^{1,2}(\mathbb{R}^n)} $$

because thanks to Lemma \ref{non-degeneracy of the solution},
$$\langle L\phi,\psi_{{k_0},j}\rangle=\langle \phi,L\psi_{{k_0},j}\rangle=0$$

with $L\phi=\phi-i^*((f'(U_{\epsilon,{k_0}})-Q(P_{k_0})g'(U_{\epsilon,{k_0}}))\phi)$.
$\\$$\\$

\noindent \underline{Control of 
$\langle N_f(\phi)| \psi_{{k_0},j}\rangle_{D^{1,2}(\mathbb{R}^n)}  +\langle N_g(\phi)|\psi_{{k_0},j}\rangle_{D^{1,2}(\mathbb{R}^n)} $:}

\medskip

 We have,

$$|\langle N_f(\phi)+N_g(\phi)|\psi_{{k_0},j}\rangle_{D^{1,2}(\mathbb{R}^n)} |\leq \|N_f(\phi)+N_g(\phi)\|_{D^{1,2}(\mathbb{R}^n)}\|\psi_{{k_0},j}\|_{D^{1,2}(\mathbb{R}^n)}\leq C \|N_f(\phi)+N_g(\phi)\|_X$$$$\leq C \|\phi\|_X^{q}\leq C'\epsilon^q$$

Nonetheless, $q>1$. Then
$$|\langle N_f(\phi)+N_g(\phi)|\psi_{{k_0},j}\rangle_{D^{1,2}(\mathbb{R}^n)} |=o(\epsilon)$$

\medskip
$\\$$\\$
\underline{Control of 
$\langle I_f(\phi)+I_g(\phi)|\psi_{{k_0},j}\rangle$}:
$\\$

We have
$$|\langle I_f|\psi_{{k_0},j}\rangle| \leq \|I_f\|_{D_{1,2}} \|\psi_{{k_0},j}\|_{D_{1,2}} \leq C \|I_f\|_X=o(\epsilon )$$

Likewise
$$|\langle I_g|\psi_{k_0,j}\rangle| = o(\epsilon)$$

$\\$$\\$
\medskip
\underline{Control of 
$\langle i^*(g'(U_{\epsilon,1})(Q(P_1)-Q_{\epsilon})\phi|\psi_{{k_0},j}\rangle$:}

By definition of $i^*$

$$\langle i^*(g'(U_{\epsilon,1})(Q(P_1)-Q_{\epsilon} )\phi)|\psi_{{k_0},j}\rangle_{D^{1,2}   (\mathbb{R}^n)}=\langle g'(U_{\epsilon,1})(Q(P_1)-Q_{\epsilon} )\phi|\psi_{{k_0},j}\rangle_{L^2   (\mathbb{R}^n)}$$
$\\$$\\$
 By Hölder's inequality,
$$|\langle i^*(g'(U_{\epsilon,1})(Q(P_1)-Q_{\epsilon} (x))\phi|\psi_{k_0,j}\rangle_{D^{1,2}   (\mathbb{R}^n)} |\leq \|\phi\|_X\|g'(U_{P_1})(Q({P_1})-Q(\epsilon x+{P_1}))\psi_{k_0,j}\|_{\frac{2n}{n+2}}.$$

But,
$$\left|\left|g'(U_{P_1})(Q(P_1)-Q(\epsilon x+P_1))\psi_{k_0,j}\right|\right|_{\frac{2n}{n+2}}\leq \epsilon Q(P_1)^a\left|\left|g'(U_{P_1})|x|\frac{\partial U_{P_1}}{\partial x_j}\right|\right|_{\frac{2n}{n+2}}.$$
Then, 
$$\langle i^*(g'(U_{P_1})(Q(P_1)-Q_{\epsilon})\phi|\psi_{k_0,j}\rangle_{D^{1,2}(\mathbb{R}^n)}  
=o(\epsilon).$$
$\\$$\\$
\underline{Control of 
$\langle i^*((g'(U_{\epsilon,1})-g'(S_{\epsilon,P}))Q_{\epsilon}\phi|\psi_{{k_0},j}\rangle$}:

$$\\$$

Set $\beta = \frac{2n}{n+2}$ and $r = \frac{d(E_1,E_2)}{2\epsilon}$. 

We have,
$$|\langle i^*((g'(S_{\epsilon,P})-g'(U_{\epsilon,k_0}))Q_{\epsilon}\phi)|\psi_{{k_0},j}\rangle|\leq C\| (g'(S_{\epsilon,P})-g'(U_{\epsilon,k_0}))\psi_{{k_0},j}\|_{\beta}~\|\phi\|_X $$

But,
$$\| (g'(U_{\epsilon,1})-g'(S_{\epsilon,P}))\psi_{1,j}\|_{\beta}^{\beta}$$
$$=\int_{\mathbb{R}^n}\left|\left[ \left(U_{P_1}(x)+U_{P_2}\left(x-\frac{P_2-P_1}{\epsilon}\right)\right)^{q-1}-(U_{P_1}(x))^{q-1}\right]\left|\frac{\partial U_{P_1}}{\partial x_j}(x)\right|\right|^{\beta} dx$$

$$ \leq 
\int_{|x |\leq r}\left|\left[\left(U_{P_1}(x)+U_{P_2}\left(x-\frac{P_2-P_1}{\epsilon}\right)\right)^{q-1}-(U^{q-1}_{P_1}(x))\right]\left|\frac{\partial U_{P_1}}{\partial x_j}(x)\right|\right|^{\beta}dx $$ $$+ 
\int_{|x |>r}\left|\left[\left(U_{P_1}(x)+U_{P_2}\left(x-\frac{P_2-P_1}{\epsilon}\right)\right)^{q-1}-(U^{q-1}_{P_1}(x))\right]\left|\frac{\partial U_{P_1}}{\partial x_j}(x)\right|\right|^{\beta}dx $$
$$ \leq 
M\left(\int_{|x|> r}\left|\frac{\partial U_{P_1}}{\partial x_j}(x)\right|^{\beta}dx + 
\int_{|x |\leq r}U_{P_2}\left(x-\frac{P_2-P_1}{\epsilon}\right)^{(q-1)\beta}\left|\frac{\partial U_{P_1}}{\partial x_j}(x)\right|^{\beta}dx\right) $$$$\leq \frac{C}{r^{(n-1)\beta-n}}+C'U^{(q-1)\beta}(r)=o(\epsilon^{\beta})$$
 (We use that $U$ is bounded and is radially decreasing). Thus,
$$\| (g'(U_{\epsilon,{k_0}})-g'(S_{\epsilon,P}))\psi_{{k_0},j}\|_{\beta}=o(\epsilon).$$

Likewise, 
$$\langle i^*((f'(S_{\epsilon,P})-f'(U_{\epsilon,k_0}))\phi| \psi_{{k_0},j}\rangle=o(\epsilon).$$
$\\$$\\$
\underline{Control of $\sum\langle i^*(g(U_{\epsilon,k})(Q_{\epsilon}-Q(P_k))|\psi_{{k_0},j}\rangle$}:
$\\$

On the one hand,

$$\langle i^*(g(U_{\epsilon,1})(Q_{\epsilon}-Q(P_1))|\psi_{{k_0},j}\rangle_{D^{1,2}(\mathbb{R}^n)}  = \int U_{P_1}^q(x)\frac{\partial U_{P_1}}{\partial x_j}(x)(Q(\epsilon x+ P_1)-Q(P_1))dx $$

But by an application of the Stokes formula (see \cite{MichelettiPistoia2003}), 
$$\int_{\mathbb{R}^n} U_{P_1}^q(x)\frac{\partial U_{P_1}}{\partial x_j}(x)(Q(\epsilon x+ P_1)-Q(P_1))dx=\frac{1}{q+1}\int_{\mathbb{R}^n}\frac{\partial U^{q+1}_{P_1}}{\partial x_j}(x)(Q(\epsilon x+ P_1)-Q(P_1))dx$$$$=-\epsilon \int_{\mathbb{R}^n} \frac{\partial Q(\epsilon x+P_1)}{\partial x_j}U^{q+1}_{P_1}(x)\frac{1}{q+1}dx=-\epsilon\frac{\partial Q}{\partial x_j}(P_1)Q(P_1)^{(q+1)a-bn} \int_{\mathbb{R}^n}U^{q+1}\frac{1}{q+1} + o(\epsilon)$$

because $U_{P}(x) = Q(P)^{a}U\left(Q(P)^bx\right)$, $\|\nabla^2Q\| < +\infty$ and $U^{q+1}|x| \in L^1$. We set $l=(q+1)a-nb$.
Thus,

$$\langle i^*(g(U_{\epsilon,1})(Q_{\epsilon}-Q(P_1))|\psi_{1,j}\rangle_{D^{1,2}(\mathbb{R}^n)}  =-\frac{\partial Q}{\partial x_j}(P_1)Q(P_1)^l C\epsilon + o(\epsilon).$$

On the other hand
$$|\langle i^*(g(U_{\epsilon,2})(Q_{\epsilon}-Q(P_2))|\psi_{1,j}\rangle|\leq \epsilon \int_{\mathbb{R}^n}U^q_{P_2}(x)|x|\left|\frac{\partial U_{P_1}}{\partial x_j}\left(x-\frac{P_1-P_2}{\epsilon}\right)\right|dx $$
$$=\epsilon \int_{\mathbb{R}^n}U^q_{P_2}(x)|x|\left|\frac{\partial U_{P_1}}{\partial x_j}\left(\frac{P_1-P_2}{\epsilon}-x\right)\right|dx$$
(because $\psi_{1,j}$ is odd).

But,

$$|\psi_{1,j}(r)| \leq C Q(P_1)^{a+b}\frac{1}{1+(Q(P_1)^br)^{n-1}} $$
Thus, 
$$\left|\psi_{1,j}\left(\frac{P_1-P_2}{\epsilon}-x\right)\right| \leq C \frac{1}{1+C_2|\frac{P_1-P_2}{\epsilon}-x|^{n-1}} $$
$$C\int_{\mathbb{R}^n}U^q_{P_2}(x)|x|\left|\frac{\partial U_{P_1}}{\partial x_j}\left(x-\frac{P_1-P_2}{\epsilon}\right)\right|dx \leq \int_{\mathbb{R}^n}U^q(C'x)|x| \frac{1}{1+C_2|\frac{P_1-P_2}{\epsilon}-x|^{{n-1}}} dx $$
$$= C\left[(U(C'.)|.|)*(\frac{1}{1+C_2|.|^{n-1}})\right]\left(\frac{P_1-P_2}{\epsilon}\right)=o(1).$$

By a standard property of convolution because $L^2*L^2 \subset C_{0}(\mathbb{R}^n)$.

We conclude that $$\left(\frac{G_{\epsilon}(P)}{Q(P_k)^l\epsilon}\right)_{k,j}$$ converges uniformly to $C\nabla \tilde{Q}$  where $C$ and $l$ are constant.

\end{proof}

To conclude we need this classical result on the Brouwer degree ( see for example \cite{AmbrosettiMalchiodi2007}). 

\begin{lemma}\label{Brouwer2}

Consider $F_{\epsilon}:\mathbb{R}^d \rightarrow \mathbb{R}^d$ continuous with $\epsilon >0$. Consider $F:\mathbb{R}^d\rightarrow\mathbb{R}^d$ continuous and a bounded open $O$ of $\mathbb{R}^d$ such that $0 \in F(O)$ but $0 \notin F(\partial O)$ and $deg(F,O,0) \ne 0$. Suppose that 
$$\|F_{\epsilon}-F\|_{L^\infty(O)}\rightarrow0.$$
Then there exists $\epsilon_1$ such that $deg(F_{\epsilon},O,0)\ne0$ for all $0<\epsilon<\epsilon_1$.

\end{lemma}
\begin{lemma}\label{solution}
Consider also $F:\mathbb{R}^d\rightarrow\mathbb{R}^d$ continuous and a bounded open $O$ such that $0 \in F(O)$ and $0 \notin F(\partial O)$. If $deg(F,O,0)\ne 0$. Then there exists $z \in O$ such that $F(z)=0$.
\end{lemma}
To conclude we need to use some classical results on the Brouwer degree. 
Set $O_m=B(\xi_1,\frac{R_1}{m})\times B(\xi_2,\frac{R_2}{m})$. By definition of $R_1$ and $R_2$ and because $\xi_1$, $\xi_2$ are nondegenerate points, $deg(\nabla\tilde{Q},O_m,0)\ne 0$ and $0 \notin \nabla \tilde{Q}(\partial O_m)$. In addition, the result \ref{Brouwer 2} show that 
$$\|F_{\epsilon}-\nabla\tilde{Q}\|_{L^\infty(O_m)}\rightarrow0$$
 Thanks to Lemma \ref{Brouwer2}, we can find $\epsilon_m$ such that $deg(F_\epsilon,O_m,0)\ne 0$ for all $\epsilon <\epsilon_m$. Using Lemma \ref{solution} we can thus build a family $(P_{\epsilon})\in (\mathbb{R}^{2n})^{(0,\epsilon_0)}$ such that
 
 \begin{itemize}
     \item for all $k \in \{1,2\}$, $P_{\epsilon,k}\rightarrow \xi_k$.
     \item $G_{\epsilon}(P_{\epsilon})=0$. 
 \end{itemize}

\begin{proposition}\label{end unbounded}

If $ W_{\epsilon}:=\sum U_{\epsilon,P_{\epsilon,k}}+\phi_{\epsilon,P_{\epsilon}}$, then 
$$W_{\epsilon}=i^*(f(W_\epsilon)-Q_{\epsilon}g(W_{\epsilon})).$$

\end{proposition}
 
\begin{proof}
By definition of $\phi$, $T_{\epsilon,P_{\epsilon}}\phi = \phi$. Thus, 

$$W_{\epsilon}-i^*(f(W_\epsilon)-Q_{\epsilon}g(W_{\epsilon}))\in K_{\epsilon,P_{\epsilon}}$$
But $G_{\epsilon}(P_{\epsilon})=0$. Thus, 
$$W_{\epsilon}-i^*(f(W_\epsilon)-Q_{\epsilon}g(W_{\epsilon}))\in K_{\epsilon,P_{\epsilon}}^\perp$$

We conclude easily because $K_{\epsilon,P_{\epsilon}}^\perp\cap K_{\epsilon,P_{\epsilon}} =\{0\}.$
\end{proof}

\subsection{Properties of the solution}
Now, we have a weak solution $W$ of the problem \eqref{Alternative} and want to show some properties of regularity of this solution. 
\begin{lemma}\label{regularity2}
For all $0<\epsilon<\epsilon_0$, $W_{\epsilon,P_{\epsilon}}\in C^2(\mathbb{R}^n)$ and $W_{\epsilon,P_{\epsilon}}$ is a strong solution of
$$-\Delta u = u^p-Q_{\epsilon}u^q.$$
\end{lemma}

\begin{proof}
We adapt the proof of Theorem 1.16 of \cite{AmbrosettiMalchiodi2007} in our case. Consider $\epsilon_0>\epsilon >0$, $x_0 \in \mathbb{R}^n$ and $R>0$. We want to show that $W_{\epsilon} \in C^2(B(x_0,r))$ for $r$ small enough. Fix $\mu>1$ such that $p\mu=\frac{n+2}{n-2}$. Set $B_m=B(x_0,\frac{R}{2^m})$. Consider the unique solution $w_m$ of the following problem  
\begin{equation*}
 \begin{cases} 
  -\Delta u=W^p-Q_{\epsilon}W^q&\text{in} \quad B_m, \\
        
        \quad u(x) = 0 &\text{on} \quad  \partial B_m.     
\end{cases}
\end{equation*} 

Set $\beta_0 =\frac{2*}{p}$. Since $W\in X$, then $W^p - Q_{\epsilon}W^q\in L^{\beta_0}(B(x_0,R))$. Then thanks to the Theorem 1.10 in \cite{AmbrosettiMalchiodi2007}, we conclude that $w_0\in L^{p_1}(B(x_0,R))$  with $p_1=\frac{n\beta_0}{n-2\beta_0}> \mu\frac{2n}{n-2}$. Set $\beta_1 =\frac{p_1}{p}$. If $2\beta_1>n$ we can stop the process. Else by definition $w_0-W$ is harmonic in $B(x_0,R)$. Thus by the Weyl's lemma, $w_0-W\in C^{\infty}(B(x_0,R))$. We conclude that $W\in L^{p_1}(\bar{B_1})$. Then by the same reasoning with $w_1$, we deduce that $W\in L^{p_2}(\bar{B_2})$ with $p_2=\frac{n\beta_1}{n-2\beta_1}> \mu p_1$. Repeat the process until $2\beta_m> n$. Then $-\Delta w_m\in L^{\beta_m}(B_m)$, then $w_m \in H^{\beta_m}$. By regularity because $2\beta_m>n$,$ w_m \in C^{0,l}(B_m)$. Then by the Weyl's lemma, $W\in C^{0,l}(\bar{B}_{m+1})$. Then by Theorem 1.10 of \cite{AmbrosettiMalchiodi2007}, $w_{m+1} \in C^2(B_{m+1})$. We conclude by the Weyl's lemma.  
\end{proof}

Now we have a strong solution $W_{\epsilon}$ of \eqref{Alternative}. We want to show that this solution is positive to answer to the initial problem.

\begin{proposition}\label{positivity}
    If $\epsilon >0$ is small enough, for all $0<\epsilon <\epsilon_0$, $W_{\epsilon,P_{\epsilon}}>0$.
    
\end{proposition}

\begin{proof}
Firstly we want to show that $W\geq 0$.
We have
$$W=i^*(f(W)-Q_{\epsilon}g(W))$$
Then, testing this equality against $W^-=-min(W,0)$, we have by definition of $f$ and $g$,
$$\langle W|W^- \rangle =0 $$
Thus $$\langle W^+|W^- \rangle-\langle W^-|W^- \rangle =0$$
$$ \langle W^-|W^- \rangle =0$$
Then $W^-=0$.
That concludes.

In addition, since $\|W\|_X \geq \|\sum U_{\epsilon,k}\|_X-\|\phi\|_X$ we can suppose that $\epsilon_0$ is chosen so that for all $\epsilon<\epsilon_0$, $W\ne0$. Note that $W$ satisfies $$-\Delta W +Q_{\epsilon}W^{q-1}W=W^p \geq 0$$

Then $$-\Delta W +cW \geq 0$$
with $c(x)=Q_{\epsilon}(x)W^{q-1}(x)\geq 0$ and $c \in L^\infty_{loc}(\mathbb{R}^n)$ thanks to \ref{regularity2}. Then by the strong maximum principle (see for example \cite{Evans2010}) because $W\ne 0$, $W(x)>0$. 

\end{proof}

Now we set $(u_{\epsilon})=W_{\epsilon}(\frac{x}{\epsilon})$. We want to show that $(u_\epsilon)$ concentrates. 
\begin{proposition}\label{regularity}
For all $0<\epsilon <\epsilon_0$, $u_{\epsilon}(x) \rightarrow0$ when $|x|\rightarrow +\infty$. 
In addition, $u_{\epsilon}$ concentrates at $(P_{\epsilon})$.

\end{proposition}

\begin{proof}
Using Theorem 4.1 of \cite{HanLin2011}, we know there is a constant $C$ such that for all $x_0 \in \mathbb{R}^n$ and for all $\epsilon$
\begin{equation}\label{controlSup}\tag{$10$}
    \|W_{\epsilon}\|_{L^\infty(B(x_0,\frac{1}{2}))} \leq C\|W_\epsilon\|_{L^r(B(x_0,1))}
\end{equation}
with $r=2^*$ (we use the fact that $W$ is a weak solution of the equation $-\Delta u+cu=0$ where $c=-W^{p-1}+Q_\epsilon W^{q-1}$ and $c$ respects $\|c\|_\beta<C$ with $\beta >\frac{n}{2}$ and $C$ independent of $\epsilon$ and $x_0$).

We deduce easily that $W(x)\rightarrow0$ when $|x|\rightarrow +\infty$ because $\|W\|_{L^r(B(x,1))}\rightarrow0  $ when $|x|\rightarrow +\infty$.

In addition, consider $R>1$. Let $E=B(\frac{P_{\epsilon,1}}{\epsilon},R)\cup B(\frac{P_{\epsilon,2}}{\epsilon},R)$. Fix $x_0 \in \mathbb{R}^n-E$. 
$$\|W_{\epsilon}\|_{L^\infty(B(x_0,\frac{1}{2}))} \leq C\left[\sum \left|\left|U_{P_{\epsilon,k}}\left(x-\frac{P_{\epsilon,k}}{\epsilon}\right)\right|\right|_{L^r(B(x_0,1))}+\left|\left|\phi\right|\right|_{L^r(B(x_0,1))}\right]$$

But, if $x\in B(x_0,1)$ then $|x-\frac{P_{\epsilon,k}}{\epsilon}|\geq R-1$. Thus, 
$$\left|\left|U_{P_{\epsilon,k}}\left(x-\frac{P_{\epsilon,k}}{\epsilon}\right)\right|\right|_r \leq \frac{C}{R^\alpha}$$

Thus, because $\|\phi_{\epsilon}\|\leq N\epsilon$,
$$\|W_{\epsilon}\|_{L^\infty(B(x_0,\frac{1}{2}))}\leq \frac{C}{R^{n-2}}+C'\epsilon $$

Conclusion:
$$sup_{y \notin E}|W_{\epsilon}(y)|\leq \frac{C}{R^{n-2}}+C'\epsilon $$
Then, by rescaling, 
$$sup_{y \notin B(P_{\epsilon,1},R\epsilon)\cup B(P_{\epsilon,2},R\epsilon)}|u_{\epsilon}(y)|\leq \frac{C}{R^{n-2}}+C'\epsilon $$

Furthermore, using Inequality \ref{controlSup}, we know that $\|\phi_\epsilon\|_\infty\leq M$ with $M$ a constant. In addition  we can find a constant $C>0$ such that  $$\|\phi_\epsilon\|_{L^\infty(B(x_0,\frac{1}{2}))}\leq C(\| \phi_{\epsilon}\|_{L^r(B(x_0,1))}+\|h_\epsilon\|_{L^\beta(B(x_0,1))}).$$
with $\beta >\frac{n}{2}$ and 
$$h_\epsilon=\left[S_{\epsilon,P_\epsilon}^p-\sum U_{\epsilon,P_\epsilon}^p+\sum Q(P_{\epsilon,k})U_{\epsilon,P_\epsilon}^q-Q_\epsilon S_{\epsilon,P_\epsilon}^q\right]$$$$+W_{\epsilon}^p-S_{\epsilon}^p-pS_\epsilon^{p-1}\phi-Q_{\epsilon}(W_{\epsilon}^q-S_{\epsilon}^q-qS_\epsilon^{q-1}\phi)$$
Indeed,  $\phi_{\epsilon}$ is weak solution of

$$-\Delta u+cu=h_{\epsilon}$$

with $$c=-pS_{\epsilon,P_\epsilon}^{p-1}+qQ_\epsilon S_{\epsilon,P_\epsilon}^{q-1} $$ 
But, 
$\|S_{\epsilon,P_\epsilon}^p-\sum U_{\epsilon,P_\epsilon}^p+\sum Q(P_{\epsilon,k})U_{\epsilon,P_\epsilon}^q-Q_\epsilon S_{\epsilon,P_\epsilon}^q\|_{L^\beta(B(x_0,1))}\leq \|S_{\epsilon,P_\epsilon}^p-\sum U_{\epsilon,P_\epsilon}^p+\sum Q(P_{\epsilon,k})U_{\epsilon,P_\epsilon}^q-Q_\epsilon S_{\epsilon,P_\epsilon}^q\|_{L^\beta(\mathbb{R}^n)}=O(\epsilon)$ (independent of $x_0$).
In addition, using the same inequalities as before, 
$$\|W_{\epsilon}^p-S_{\epsilon}^p-pS_\epsilon^{p-1}\phi_{\epsilon}-Q_{\epsilon}(W_{\epsilon}^q-S_{\epsilon}^q-qS_\epsilon^{q-1}\phi_{\epsilon})\|_{L^\beta(B(x_0,1))}\leq C\|\phi_{\epsilon}^q\|_{L^\beta(B(x_0,1))}$$

$$\leq C\|\phi_{\epsilon}\|_{L^\infty(B(x_0,1))}\|\phi_{\epsilon}^{q-1}\|_{L^\beta(B(x_0,1))}\leq CM\|\phi_{\epsilon}^{q-1}\|_{L^\beta(B(x_0,1)}\leq C'\|\phi_{\epsilon}\|_X.$$
because $\beta$ can be chosen so that $s\leq (q-1)\beta\leq 2^*$. Because $\|\phi_{\epsilon}\|_X\leq C\epsilon$, we conclude for $\epsilon$ small enough 
$$\|\phi_{\epsilon}\|_{L^\infty(B(x_0,1))}= O(\epsilon)$$
with $O$ independent of $x_0$. Thus,
$sup_{x\in B(P_{\epsilon,k},\frac{\epsilon}{2})}\left|\phi_{\epsilon}\left(\frac{x}{\epsilon}\right)\right|\leq sup_{x\in B(\frac{P_{\epsilon,k}}{\epsilon},\frac{1}{2})}\left|\phi_{\epsilon}\left(x\right)\right|\leq C (\| \phi_{\epsilon}\|_{L^r(B(\frac{P_{\epsilon,k}}{\epsilon} ,1))}+\|h_\epsilon\|_{L^\beta(B(\frac{P_{\epsilon,k}}{\epsilon},1))})\rightarrow 0$. 

We conclude especially $\phi_\epsilon(P_{\epsilon,k_0})\rightarrow0$. In addition, $\sum U_{P_{\epsilon,k}}\left(\frac{P_{\epsilon,k_0}-P_{\epsilon,k}}{\epsilon}\right)\rightarrow U_{\xi_{k_0}}(0)$.

Then, $u_\epsilon(P_{\epsilon,{k_0}})\rightarrow  U_{\xi_{k_0}}(0)>0$.

\end{proof}

Using \ref{regularity}, we conclude the proof of Theorem~\ref{MainWholeSpace}.
\medskip
\begin{remark}
 All this study can be indeed be extend to an arbitrary $K$.
Let $K\ge1$ be fixed and let $\xi_1,\dots,\xi_K$ be distinct nondegenerate
critical points of $Q$. Choose $R_k>0$ such that $\xi_k$ is the only critical
point of $Q$ in $\overline B(\xi_k,R_k)$, such that the balls
$\overline B(\xi_k,R_k)$ are pairwise disjoint (and contained in $\Omega$ in
the bounded case), and set
\[
E:=\prod_{k=1}^{K}\overline B(\xi_k,R_k),
\qquad
\rho:=\min_{k\ne k'}d\big(\overline B(\xi_k,R_k),\overline B(\xi_{k'},R_{k'})\big)>0 .
\]
The whole construction carries over word for word, with $P\in(\mathbb{R}^n)^K$,
$S_{\epsilon,P}=\sum_{k=1}^{K}U_{\epsilon,k}$ and
$K_{\epsilon,P}=\mathrm{Span}\{\partial_jU_{\epsilon,k}:1\le k\le K,\,1\le j\le n\}$,
now of dimension $Kn$. The only points deserving a comment are the following.
\begin{itemize}
\item[(a)] \emph{The Gram matrix.} $G_{\epsilon,P}$ becomes a $Kn\times Kn$ matrix,
made of $K^2$ blocks of size $n\times n$. Its diagonal blocks are
$\big(\langle\partial_jU_{P_k}|\partial_{j'}U_{P_k}\rangle\big)_{j,j'}$, which
converge to $\lambda_k\mathrm{Id}_n$ with $\lambda_k>0$ by radial symmetry;
its off-diagonal blocks tend to $0$ as $\epsilon\to0$, uniformly in $P\in E$,
because $|P_k-P_{k'}|/\epsilon,\ge\rho/\epsilon,\to+\infty$ for $k\ne k'$. Hence $G_{\epsilon,P}$ is invertible with
uniformly bounded inverse --- for which one writes, as in the case $K=2$,
$G_{\epsilon,P}^{-1}=(D_{\epsilon,,P}+G_{\epsilon,P}-D_{\epsilon,P})^{-1}$ with $D_{\epsilon,P}$ the
block-diagonal part --- and Lemma \ref{Projection prop} holds unchanged, with the same proof.
 
\item[(b)] \emph{The interaction estimates.} All the terms measuring the
interaction between distinct peaks ($I_f$, $I_g$, $I_{f'}$, $I_{g'}$, and the
terms $\langle i^*(g(U_{\epsilon,k})(Q_\epsilon-Q(P_k)))|\psi_{k_0,j}\rangle$ with
$k\ne k_0$) are estimated pairwise: for $e\ge1$ the elementary inequality
\[
\Big|\Big(\sum_{k=1}^{K}a_k\Big)^{e}-\sum_{k=1}^{K}a_k^{e}\Big|
\le C(K,e)\sum_{k\ne k'}\big(a_k^{e-1}a_{k'}+a_{k'}^{e}\big),
\qquad a_1,\dots,a_K\ge0,
\]
reduces each of them to a finite sum of at most $K(K-1)$ terms of exactly the
form treated in Sections 3.3.3 and 3.4, with $\rho$ in place of
$d(E_1,E_2)$ and $r=\frac{\rho}{2\epsilon,}$. Since $K$ is fixed, the resulting
bounds --- $o(\epsilon)$ for $I_f,I_g$, $o(1)\|\phi\|_X$ for $I_{f'},I_{g'}$ --- are
unchanged.
 
\item[(c)] \emph{The finite-dimensional reduction.} The map $G_\epsilon,$ now takes
values in $\mathbb{R}^{Kn}$, and Proposition \ref{Brouwer 2} gives that
$F_\epsilon,\big(CG_{\epsilon,k,j}(P)Q(P_k)^{-l}\epsilon^{-1}\big)_{k,j}$ converges
uniformly on $O_m=\prod_{k}B(\xi_k,\frac{R_k}{m})$ to $\nabla\tilde Q$,
where now $\tilde Q(P)=\sum_{k=1}^{K}Q(P_k)$. Lemmas \ref{Brouwer2} and \ref{solution}
then apply as before.
 
\item[(d)] \emph{Section 3.5} is unchanged, with
$E=\bigcup_{k=1}^{K}B\big(\frac{P_{\epsilon,k}}{\epsilon},R\big)$ in the proof of
Proposition \ref{regularity}.
\end{itemize}

\
\end{remark}
\newpage 
\begin{center}
\section{Bounded case}
\end{center}
In this section, we suppose $\Omega$ is a bounded and smooth open set and $q_1<q<p<\frac{n+2}{n-2}$. As before, for clarity, we restrict ourselves to the case $K=2$. We aim to prove the following Theorem
$\\$$\\$

\begin{theorem}\label{Main bounded}
Suppose $q\in(q_1,p)$ and $n\geq9$.
Suppose that $\xi_1,\xi_2\in \Omega$ are $2$ different
non-degenerate critical points of $Q$. Then there exists $\epsilon_0 > 0$ such that, for all $0 < \epsilon < \epsilon_0$,
problem \eqref{Main} admits a strong positive solution $u_{\epsilon}$ which concentrates at $(\xi_k)_{k \in \{1,2\}}$.

\end{theorem}

We use the same strategy as before: find a suitable manner to rewrite the equation; solve the auxiliary equation with the Banach fixed-point theorem with parameter; solve the bifurcation using the Brouwer degree; check that the function found has the expected properties. Thus, this section will have the same structure as the section before and we will only dwell on the differences between them.  
There are some important differences with the unbounded case. Indeed, our rescaling problem is in this case 

\begin{equation*}
 \begin{cases} 
    -\Delta u(x)=u^p(x)-Q(\epsilon x)u^q(x) \quad &\text{in} \quad \Omega_{\epsilon}, \\
        \quad u > 0  \quad &\text{in}  \quad \Omega_{\epsilon},\\
      \quad u=0   &\text{on}\quad \partial\Omega_{\epsilon}, 
     \end{cases}
\end{equation*} 
where $\Omega_{\epsilon}=\frac{\Omega}{\epsilon}$. We are so looking for a solution belonging at least to $H_0^1(\Omega_{\epsilon})$. Thus we cannot work with $i^*$  because $i^*(v)$ does not necessarily belong to this space even if $v $ is null outside $\Omega_{\epsilon}$. We need thus to replace $i^*$ by a suitable function. This is the goal of the next section. In the same spirit, $U$ has to be replaced by a suitable function of $H_0^1(\Omega_{\epsilon})$. This is the goal of Section 3.2.
\subsection{The functional framework}
Set $\Omega_{\epsilon}=\frac{\Omega}{\epsilon}$.

In this section, we will work with the space: $ H^1_0(\Omega_{\epsilon})$ given with the scalar product
$$\langle u,v\rangle_{\epsilon}= \int_{\Omega_{\epsilon}}\nabla u(x)\nabla v(x) dx.$$

As before we choose to work with 
$X_{\epsilon}:=L^s(\Omega_{\epsilon})\cap H^1_0(\Omega_{\epsilon})$ associated with the norm $$\|u\|_{X_\epsilon}= max\left\{\|u\|_{H^{1}_0(\Omega_{\epsilon})},\|u\|_{L^s(\Omega_\epsilon)}\right\}.$$ In the following we will identify $\phi \in H^1_0(\Omega_{\epsilon})$ with the extension by $0$ of $\phi$ on $\mathbb{R}^n-\Omega_{\epsilon}$. Thus, $X_{\epsilon}\subset X$. In addition, we have that $$\forall \phi \in X,N_\epsilon(\phi):=\max\left\{\|\nabla\phi\|_{L^2(\Omega_\epsilon)},\|\phi\|_{L^s(\Omega_\epsilon)}\right\}\leq \|\phi\|_X.$$ and 
$$\forall \phi \in X_{\epsilon},\|\phi\|_{X_{\epsilon}}=\|\phi\|_X.$$
In the following, we will not distinguish $N_\epsilon$ and $\|.\|_{\epsilon}$.
\medskip

We will "replace" $i$ by the inclusion $i_{\epsilon}: H_0^1(\Omega_{\epsilon})\rightarrow L^{\frac{2n}{n-2}}(\Omega_{\epsilon})$ which is continuous.

$$\\$$
As before we can define the adjoint $i_{\epsilon}^*:L^\frac{2n}{n+2}(\Omega_{\epsilon})\rightarrow H^{1}_0(\Omega_{\epsilon})$ which will be useful for our study 

$$i_{\epsilon}^*(v)= u \iff \forall \phi \in H_0^1(\Omega_{\epsilon}), \langle u,\phi\rangle_{\epsilon}=\int_{\Omega_{\epsilon}} v\phi.$$

\noindent As in Section 3, we need some control on $i_\epsilon^*$. The next lemma shows this control can even be independent $\quad$ of $\epsilon$.
\begin{lemma}\label{control L^s bounded}
$\\$$\\$
We have the following inequalities 
\begin{itemize}
    \item There exists $C>0$ (which does not depend on $\epsilon$) such that, 
    $$\forall v\in L^{\frac{2n}{n+2}}(\Omega_{\epsilon}),\|i^*_{\epsilon}(v)\|_{H^1_0(\Omega_{\epsilon})}\leq C\|v\|_{L^\frac{2n}{n+2}(\Omega_{\epsilon})}.$$
    \item There exists $C>0$ (which does not depend on $\epsilon$) such that for all $v \in L^{\frac{2n}{n+2}}(\Omega_{\epsilon})\cap L^{{\frac{ns}{n+2s}}}(\Omega_{\epsilon})$ then $i^*(v) \in L^s(\Omega_{\epsilon})$ and $\|i_{\epsilon}^*(v)\|_s \leq C\|v\|_{\frac{sn}{n+2s}}$.
\end{itemize}

\end{lemma}

\begin{proof}
The first point is clear because the constant in the Gagliardo—Nirenberg—Sobolev inequality does not depend on the domain. The second point is a consequence of the fact that the Green's function for the Laplacian on a bounded domain is bounded by the green function on $\mathbb{R}^n$ and the Hardy-Littlewood-Sobolev inequality does not depend on $\epsilon$. 
\end{proof}

\begin{remark}

The advantage of this space is that we can control by interpolation all spaces $L^r(\Omega_{\epsilon})$ with $s\leq r\leq 2^*$ with some constants which do not depend on $\epsilon$.

\end{remark}

\subsection{Preliminaries }

In this section, we will note: $U_{\epsilon,P}(x)= U_P(x-\frac{P}{\epsilon})$. Contrary to the previous case, we will not  directly use $U_{\epsilon,P}$ because $U_{\epsilon,P} \notin H_0^1(\Omega_{\epsilon})$.
That is why we now define $V_{\epsilon,P}$ the projection of $U_{\epsilon,P}$ as follows: 

$\\$

$V_{\epsilon,P} \in H^1_0(\Omega_{\epsilon})$ is the unique solution of the problem
\begin{equation*}\label{Projection} 
\begin{cases} 
   -\Delta u(x)=U_{\epsilon,P}^p(x)-Q(P)U_{\epsilon,P}^q(x) \quad &\text{in} \quad \Omega_{\epsilon}, \\
        
        \quad u= 0, \quad &\text{on} \quad  \partial \Omega_{\epsilon}. 
\end{cases}
\end{equation*} 

Note that $V_{\epsilon,P}\in C^2(\bar\Omega_{\epsilon})$ because $\Omega$ and $U$ are smooth.
\medskip

Recall that $U_{\epsilon,P}$ solves the following problem
\begin{equation*}\label{Projection} 
 \begin{cases} 
  -\Delta u(x)=U_{\epsilon,P}^p(x)-Q(P)U_{\epsilon,P}^q(x)&\text{in} \quad \Omega_{\epsilon}, \\
        
        \quad u(x) = U_{\epsilon,P}(x) &\text{on} \quad  \partial \Omega_{\epsilon}. 
\end{cases}
\end{equation*}

Then, $U_{\epsilon,P}-V_{\epsilon,P}$ solves the problem 

\begin{equation*}
 \begin{cases} 
  \Delta u(x)=0&\text{in} \quad \Omega_{\epsilon}, \\
        
        \quad u(x) = U_{\epsilon,P}(x) &\text{on} \quad  \partial \Omega_{\epsilon}. 
\end{cases}
\end{equation*}

Therefore, by the maximum principle (because $U,V$ and $\Omega$ are regular enough), in $\Omega_{\epsilon}$,

$$|U_{\epsilon,P}-V_{\epsilon,P}|\leq max_{x \in \partial\Omega_{\epsilon}} U_{\epsilon,P}(x) $$

But recall that
$$|U_P(x)|\leq C(P)\frac{1}{1+|x|^\alpha}$$
Thus,
$$max_{x \in \partial\Omega_{\epsilon}} U_{\epsilon,P}(x)\leq C(P)max_{x \in \partial \Omega_{\epsilon}}\frac{1}{1 +|x-\frac{P}{\epsilon}|^\alpha}\leq C(P)max_{x \in \partial \Omega_{\epsilon}}\frac{\epsilon^\alpha}{\epsilon^\alpha +d(\partial\Omega,P)^\alpha}\leq C\epsilon^{\alpha}$$
where $C$ a constant which does not depend on $P$ if $P$ belongs to some compact set (that will be always the case).

Thus,

$$|U_{\epsilon,P}-V_{\epsilon,P}|\leq C\epsilon^\alpha$$

\noindent Furthermore, we also consider the projection $\psi_{\epsilon,P,j}$ of $\frac{\partial U_{\epsilon,P}} {\partial x_j}$ as follows

$\psi_{\epsilon,P,j}\in H^1_0(\Omega_{\epsilon})$ is the unique solution of the problem

\begin{equation*}\label{Projection derivation} 
\begin{cases} 
  -\Delta u=\left(pU_{\epsilon,P}^{p-1}(x)-Q(P)qU_{\epsilon,P}^{q-1}\right)\frac{\partial U_{\epsilon,P}}{\partial x_j}&\text{in} \quad \Omega_{\epsilon}, \\
        
        \quad u(x) = 0&\text{on} \quad  \partial \Omega_{\epsilon}.
\end{cases}
\end{equation*}

Thus, as before, in $\Omega_{\epsilon}$,

$$\left|\psi_{\epsilon,P,j}-\frac{\partial U_{\epsilon,P}}{\partial x_j}\right| \leq C\epsilon^{\alpha+1}.$$

We also need the following results

\begin{lemma}\label{Continuity P}
Consider $\epsilon >0$. Then, 
\begin{equation*}
\begin{cases}

\Omega \rightarrow X_{\epsilon}\\
 P \mapsto V_{\epsilon,P}
\end{cases}
\end{equation*}

and 
\begin{equation*}
\begin{cases}

\Omega \rightarrow X_{\epsilon}\\
 P \mapsto \psi_{\epsilon,P,j}
\end{cases}
\end{equation*}

\end{lemma}
are continuous.
\begin{proof}
Notice $i_{\epsilon}^*$ is continuous and 

$$P\mapsto f(U_{\epsilon,P})-Q(P) g(U_{\epsilon,P})\quad \text{and} \quad P\mapsto  (f'(U_{\epsilon,P})-Q(P) g'(U_{\epsilon,P}))\frac{\partial U_{\epsilon,P}}{\partial x_j}$$

are continuous.
\end{proof}

\begin{lemma}\label{Convergence Projection}
Consider $E$ a compact set of $\Omega$. Consider sequences $(\epsilon_m)\in (\mathbb{R_+^*})^{\mathbb{N}}$ and $(A_m)\in \Omega^{\mathbb{N}}$ such that $\epsilon_m \rightarrow 0$ and $A_m \rightarrow A\in \Omega$. 
Then, 
$$\left|\left|\psi_{\epsilon_m,m,j}-\frac{\partial U_{\epsilon_m,A_m}}{\partial x_j}\right|\right|_{X}\rightarrow 0.$$

\end{lemma}
\begin{proof}

It is quite obvious that $\left|\left|1_{\mathbb{R}^n-\Omega_{\epsilon_m}}\frac{\partial U_{\epsilon_m,A_m}}{\partial x_j}\right|\right|_X\rightarrow0$.
\smallskip
Furthermore, one the one hand,
$$\left|\left|\psi_{\epsilon_m,m,j}-\frac{\partial U_{\epsilon_m,A_m}}{\partial x_j}\right|\right|_{L^s(\Omega_{\epsilon_m})}\leq \left|\left|\psi_{\epsilon_m,m,j}-\frac{\partial U_{\epsilon_m,A_m}}{\partial x_j}\right|\right|_{L^\infty(\Omega_{\epsilon_m})}mes(\Omega_{\epsilon_m})^\frac{1}{s}\leq C\epsilon_m^{n-1}mes(\Omega_{\epsilon_m})^{\frac{1}{s}}$$

But, $\Omega$ is bounded, then we can find $l>0$ such that $\Omega\subset B(0,l)$ then, $\Omega_{\epsilon_m}\subset B(0,\frac{l}{\epsilon_m})$ then $mes(\Omega_{\epsilon_m})^{\frac{1}{s}}\leq M\epsilon_m^{-\frac{n}{s}}$. 
Then, 
$$\left\|\psi_{\epsilon_m,m,j}-\frac{\partial U_{\epsilon_m,A_m}}{\partial x_j}\right\|_{L^s(\Omega_{\epsilon_m})}\leq M\epsilon_m^{n-1-\frac{n}{s}}=o(1)$$
Because $s>\frac{n}{n-2}>\frac{n}{n-1}$.

$\\$$\\$
On the other hand, 
$$\left\|\psi_{\epsilon_m,m,j}-\frac{\partial U_{\epsilon_m,A_m}}{\partial x_j}\right\|_{D^{1,2}(\mathbb{R}^n)}^2=\left \langle\psi_{\epsilon_m,m,j}-\frac{\partial U_{\epsilon_m,A_m}}{\partial x_j}|\psi_{\epsilon_m,m,j}-\frac{\partial U_{\epsilon_m,A_m}}{\partial x_j} \right\rangle$$
$$=\left\langle\psi_{\epsilon_m,m,j}|\psi_{\epsilon_m,m,j}-\frac{\partial U_{\epsilon_m,A_m}}{\partial x_j} \right\rangle-\left\langle\frac{\partial U_{\epsilon_m,A_m}}{\partial x_j}|\psi_{\epsilon_m,m,j}-\frac{\partial U_{\epsilon_m,A_m}}{\partial x_j} \right\rangle$$

but because $\Delta(\psi_{\epsilon_m,m,j}-\frac{\partial U_{\epsilon_m,A_m}}{\partial x_j})=0$ and $\psi_{\epsilon_m,m,j}\in H_0^1(\Omega_{\epsilon_m})$ in the weak sense,

$$\left\langle\psi_{\epsilon_m,m,j}|\psi_{\epsilon_m,m,j}-\frac{\partial U_{\epsilon_m,A_m}}{\partial x_j} \right\rangle_{X_{\epsilon_m}}=0$$

In addition, 
$$\left\langle\frac{\partial U_{\epsilon_m,A_m}}{\partial x_j}|\psi_{\epsilon_m,m,j}-\frac{\partial U_{\epsilon_m,A_m}}{\partial x_j} \right\rangle=\left\langle\frac{\partial U_{\epsilon_m,A_m}}{\partial x_j}|\psi_{\epsilon_m,m,j}\right\rangle-\left\langle\frac{\partial U_{\epsilon_m,A_m}}{\partial x_j}|\frac{\partial U_{\epsilon_m,A_m}}{\partial x_j} \right\rangle$$

Then, by the Stokes Theorem
$$\left\langle\frac{\partial U_{\epsilon_m,A_m}}{\partial x_j}|\psi_{\epsilon_m,m,j}\right\rangle_{X_{\epsilon_m}}=-\left\langle\Delta\frac{\partial U_{\epsilon_m,A_m}}{\partial x_j}|\psi_{\epsilon_m,m,j}\right\rangle_2$$

Likewise, 
$$\left\langle\frac{\partial U_{\epsilon_m,A_m}}{\partial x_j}|\frac{\partial U_{\epsilon_m,A_m}}{\partial x_j} \right\rangle_{X_{\epsilon_m}}$$
$$=-\left\langle \Delta\frac{\partial U_{\epsilon_m,A_m}}{\partial x_j}|\frac{\partial U_{\epsilon_m,A_m}}{\partial x_j} \right\rangle_2+\int_{\partial\Omega_{\epsilon_m}}\frac{\partial U_{\epsilon_m,A_m}}{\partial x_j} \frac{\partial (\frac{\partial U_{\epsilon_m,A_m}}{\partial x_j})}{\partial v}dS $$

But, 
$$\left|\left\langle\Delta\frac{\partial U_{\epsilon_m,A_m}}{\partial x_j}|\psi_{\epsilon_m,m,j}\right\rangle_2-\left\langle \Delta\frac{\partial U_{\epsilon_m,A_m}}{\partial x_j}|\frac{\partial U_{\epsilon_m,A_m}}{\partial x_j} \right\rangle_2\right|$$
$$\leq \left|\left|\Delta\frac{\partial U_{\epsilon_m,A_m}}{\partial x_j}\right|\right|_1\left|\left|\psi_{\epsilon_m,m,j}-\frac{\partial U_{\epsilon_m,A_m}}{\partial x_j}\right|\right|_{\infty}\rightarrow0$$
In addition, 
$$\left|\int_{\partial\Omega_{\epsilon_m}}\frac{\partial U_{\epsilon_m,A_m}}{\partial x_j} \frac{\partial(\frac{\partial U_{\epsilon_m,A_m}}{\partial x_j})}{\partial v}dS\right|\leq H^{n-1}(\partial\Omega_{\epsilon_m})\left|\left|\frac{\partial U_{\epsilon_m,A_m}}{\partial x_j} \frac{\partial(\frac{\partial U_{\epsilon_m,A_m}}{\partial x_j})}{\partial v}\right|\right|_\infty \rightarrow0. $$
\end{proof}
\subsection{Solving the auxiliary equation}
For a fixed $\epsilon > 0$, we want to find a solution with the form $$W_{\epsilon,P}=  \sum V_{\epsilon,k} + \phi_{\epsilon,P} $$ with $\phi_{\epsilon,P}\in K_{\epsilon,P}^\perp =\left\{v \in X_\epsilon, \langle v|\psi_{\epsilon,k,j}\rangle_{\epsilon}=0\right \}$, $\epsilon$ is small enough and  $P \in \Omega^2$. Let $S_{\epsilon,P}:=\sum V_{\epsilon,k}$.

\noindent Equation \eqref{Main} becomes
\begin{equation}\label{BVP4}\tag{$11$}
 \begin{cases} 
    -\Delta u(x)=u^p(x)-Q(\epsilon x)u^q(x) \quad &\text{in} \quad \Omega_{\epsilon}, \\
        \quad u > 0  \quad &\text{in}  \quad \Omega_{\epsilon},\\
      \quad u=0   &\text{on}\quad \partial\Omega_{\epsilon}. 
\end{cases}
\end{equation} 

\noindent As before we firstly study the equation 
\begin{equation*}\label{Main-single-Bounded}
W = i_{\epsilon}^*(f(W) - g(W)Q_{\epsilon})
\end{equation*}

\subsubsection{Decomposition of the equation}
    
We rewrite as follows (using $\phi$ to mean  $\phi_{\epsilon,P}$):

$$W=N_f(\phi) + N_g(\phi) + I_f + I_g+ I_{f'} + I_{g'} +\sum i_{\epsilon}^*(g'(V_{\epsilon,k})\phi(Q(P_k)-Q_{\epsilon}))+\sum i_{\epsilon}^*((f'(V_{\epsilon,k})-g'(V_{\epsilon,k})Q(P_k))\phi)+$$$$  \sum i_{\epsilon}^*(f(V_{\epsilon,k})-Q(P_k)g(V_{\epsilon,k}))+\sum i_{\epsilon}^*((Q(P_k)-Q_\epsilon) g(V_{\epsilon,k})) $$

Where
$$
\left\{
\begin{array}{ll}
      N_f(\phi)=i_{\epsilon}^*((f(S_{\epsilon,P}+\phi)-f(S_{\epsilon,P})-f'(S_{\epsilon,P})\phi))\\
        N_g(\phi) = -i_{\epsilon}^*(Q_{\epsilon}[g(S_{\epsilon,P}+\phi)-g(S_{\epsilon,P})-g'(S_{\epsilon,P})\phi])\\
        
       I_f= i_{\epsilon}^*(f(S_{\epsilon,P})- \sum f(V_{\epsilon,k}))\\ I_g= -i_{\epsilon}^*(Q_{\epsilon}(g(S_{\epsilon,P})- \sum g(V_{\epsilon,k})))\\
       I_{f'}=i_{\epsilon}^*((f'(S_{\epsilon,P}) - \sum f'(V_{\epsilon,k}))\phi) \\
       I_{g'}= -i_{\epsilon}^*((g'(S_{\epsilon,P}) - \sum g'(V_{\epsilon,k}))Q_{\epsilon}\phi)\\   
\end{array}     
\right.
$$

Hence, because $S_{\epsilon,P}= \sum i_{\epsilon}^*(f(U_{\epsilon,k})-Q(P_k)g(U_{\epsilon,k}))$

$$\phi = N_f(\phi) + N_g(\phi) + I_f + I_g+ I_{f'} + I_{g'} +\sum i_{\epsilon}^*(g'(V_{\epsilon,k})\phi(Q(P_k)-Q_{\epsilon}))+\sum i_{\epsilon}^*((f'(V_{\epsilon,k})-g'(V_{\epsilon,k})Q(P_k))\phi)+$$$$  +\sum i_{\epsilon}^*((Q(P_k)-Q_\epsilon) g(V_{\epsilon,k})) +  \sum i_{\epsilon}^*(f(V_{\epsilon,k})-f(U_{\epsilon,k})-Q(P_k)(g(V_{\epsilon,k})-g(U_{\epsilon,k}))) $$

$$$$

Set
$$L_{\epsilon,P}\phi = \phi -  \sum i_{\epsilon}^*((f'(U_{\epsilon,k}) -g'(U_{\epsilon,k})Q(P_k))\phi)$$
 Therefore

 $$L_{\epsilon,P}\phi = N_f(\phi) + N_g(\phi) + I_f + I_g+ I_{f'} + I_{g'} $$$$  +\sum i_{\epsilon}^*(g'(V_{\epsilon,k})\phi(Q(P_k)-Q_{\epsilon}))+\sum i_{\epsilon}^*((Q(P_k)-Q_\epsilon) g(V_{\epsilon,k})) +  R_1 +R_2$$

 with 
$$\left\{
\begin{array}{ll}
      R_1=\sum i_{\epsilon}^*(f(V_{\epsilon,k})-f(U_{\epsilon,k})-Q(P_k)(g(V_{\epsilon,k})-g(U_{\epsilon,k}))) \\
        R_2=\sum i_{\epsilon}^*((f'(V_{\epsilon,k})-f'(U_{\epsilon,k})-[g'(V_{\epsilon,k})-g'(U_{\epsilon,k})]Q(P_k))\phi)
        \end{array}
\right. $$
\subsubsection{Study of $L_{\epsilon,P}$}
As before, the following lemmas are true

\begin{lemma}\label{Continuity 3}
$\\$
$L_{\epsilon,P}$ satisfies:
\begin{itemize}
    \item $L_{\epsilon,P}\phi \in X, \forall \phi \in X$ \item $L_{\epsilon,P}$ is continuous on $X$.
    \end{itemize}
\end{lemma}

\begin{lemma}\label{Compact Bouded}
$\\$

 $Id-L_{\epsilon,P}$ is a compact operator.

\end{lemma}

\begin{lemma}

Consider two compact sets $E_1$ and $E_2$  of $~\mathbb{R}^{n}$ such that $d(E_1,E_2)>0$ and set $E:=E_1\times E_2$. 
Denote $\tau_{\epsilon,P}:\phi \in X\mapsto \phi(. +\frac{P}{\epsilon}) $. One has
\begin{itemize}
   
    \item Consider a sequences $(\epsilon_m)\in (\mathbb{R}_+^*)^\mathbb{N}$ and $P_m \in E^\mathbb{N}$ such that $\epsilon_m \rightarrow 0$ and $P_m \rightarrow (A_1,A_2)$. Then, $||\tau_{\epsilon_m,P_{m,k}}\Pi_{\epsilon_m,P_m}\tau_{\epsilon_m,P_{m,k}}^{-1}\eta - \Pi_{\infty,A_k}\eta||_X\rightarrow0$  where $\Pi_{\infty,A_k}$ is the orthogonal projection on $ K_{\infty,A_k}^{\perp}$ where  $K_{\infty,A_k}=Span\{\frac{\partial U_{A_k}}{\partial x_j},j\in\{1,...,n\}\}$.
     \item There exists $\epsilon_0 >0$, there exists $C>0$ such that $\forall 0<\epsilon <\epsilon_0,\forall P\in E, ||\Pi_{\epsilon,P}||_{X_{\epsilon}}\leq C$. 
    \item There exists $\epsilon_0$ such that for all $\epsilon <\epsilon_0$, $P \mapsto \Pi_{\epsilon,P}$ is continuous.
    
    \end{itemize}

\end{lemma}
\begin{proof}
It suffices to use the lemma \ref{Continuity P} and \ref{Convergence Projection}.
\end{proof}
As before, the following result show that $\tilde{L}_{\epsilon,P}$ is invertible if we consider the convenient spaces. 
\begin{proposition}
    
Consider two compact sets $E_1$, $E_2$ of $\Omega$  such that $d(E_1,E_2)>0$ and set $E:=E_1\times E_2$. There exists $C$ such that

$$\forall P\in E, \epsilon <\epsilon_0,\forall\phi \in K_{\epsilon,P}^\perp, ||\tilde{L}_{\epsilon,P} \phi||_X \geq C||\phi||_X.$$

\end{proposition}
\begin{proof}
    
$\\$$\\$
The proof is quite similar to the proof of the proposition \ref{Uniformity}. We only indicate the most important modifications. 
$\\$
The most important differences are 
\begin{itemize}
    \item We work with $X_{\epsilon}$. Nonetheless, $X_{\epsilon}\subset X$. Thus, we can apply the reflexivity as before.
    \item As before, we need to apply $\Pi_{\epsilon,P_{m,k_0}}$ to $\tau^{-1}\eta$ with $\eta \in C_c^{\infty}(\mathbb{R}^n)$ but $\tau^{-1}\eta$ does not belong necessarily to $X_{\epsilon}$. But there exists $m_1$ such that $\forall m >m_1, supp(\eta)+\frac{P_{m,{k_0}}}{\epsilon_m}\subset \Omega_{\epsilon_m}  $.
Indeed, $\Omega_{\epsilon_m}=\frac{\Omega}{\epsilon_m}$. 

In addition, because $supp(\eta)$ is compact,  $$\left|\left|\epsilon_m\left( supp(\eta)+\frac{P_{m,{k_0}}}{\epsilon_m}-A_{k_0}\right)\right|\right|_{L^\infty(supp(\eta))}\rightarrow0$$
We conclude because $A_{k_0}\in \Omega$ and $\Omega$ is open. 
Then for $m>m_1$, $\tau^{-1}\eta \in X_{\epsilon}$. 
 
 \item In the bounded case, $\psi_{\epsilon_m,k_0,j}$ replaces $\frac{\partial U_{\epsilon_m,P_{m,k_0}}}{\partial x_j}$. But 
On $\Omega_{\epsilon}, \Delta(\psi_{{k_0},j,\epsilon}-\frac{\partial U_{\epsilon,P_{m,k_0}}}{\partial x_j})=0$. Then, 
$$\left\langle \phi_{m}\big| \frac{\partial U_{\epsilon_m,P_{m,k_0}}}{\partial x_j}\right\rangle_{X}=\left\langle \phi_{m}\big| \psi_{\epsilon_m,{k_0},j}\right\rangle_{X_{\epsilon_m}}=0$$
Then as before,$$\left\langle \phi_{m}\big| \psi_{\epsilon_m,{k_0},j}\right\rangle_{X_{\epsilon_m}}\rightarrow\left\langle \psi_{{k_0}}\big| \frac{\partial U_{A_{k_0}}}{\partial x_j}\right\rangle_{X}=0 $$
 
 \end{itemize}

\end{proof}
As before, we deduce the lemma 
\begin{corollary}\label{inversibility bounded}
$\\$
Consider a compact set $E_1\times E_2$ of $\Omega^{2}$ such that $d(E_1,E_2)>0$,
$\tilde{L}_{\epsilon,P}$ is invertible for all $P \in E$ and its inverse is continuous. In addition, there exists $C$ a constant $\|\tilde{L}^{-1}_{\epsilon,P}\|_{K_{\epsilon,P}^\perp}\leq C$ for all $\epsilon <\epsilon_0$ and $P \in E$. 

\end{corollary}

$$\\$$

Therefore, the equation can be written as follows:

 $$L_{\epsilon,P}\phi = N_f(\phi) + N_g(\phi) + I_f + I_g+ I_{f'} + I_{g'} $$$$  +\sum i_{\epsilon}^*(g'(V_{\epsilon,k})\phi(Q(P_k)-Q_{\epsilon}))+i_{\epsilon}^*((Q(P_k)-Q_\epsilon) g(V_{\epsilon,k})) +  R_1 +R_2$$

thus  $$\phi = \tilde{L}_{\epsilon,P}^{-1}\Pi_{\epsilon,P}(N_f(\phi) + N_g(\phi) + I_f + I_g+ I_{f'} + I_{g'} $$$$  +\sum i_{\epsilon}^*(g'(V_{\epsilon,k})\phi(Q(P_k)-Q_{\epsilon}))+i_{\epsilon}^*((Q(P_k)-Q_\epsilon) g(V_{\epsilon,k})) +  R_1 +R_2)$$

Thus this equation takes the form: 
$\phi = T_{\epsilon,P}(\phi)$.
To apply the Banach fixed-point theorem, one needs some controls on the terms which play a role in this equation. That is the goal of the next section.
\subsubsection{Control of terms}

\begin{lemma}\label{Stability3}
$\\$
Let $E_1$ and $E_2$ two compacts such that $d(E_1,E_2)>0$ and set $E=E_1\times E_2$.
There exists $N >0$ such that for $\epsilon$ small enough, $B(0,N\epsilon)$ is stable under $T_{\epsilon,P}$ for all $P \in E$.
$$\\$$
\end{lemma}

In order to simplify the proof, we need the result
\begin{lemma}\label{Control bounded}
We have
$$\|i_{\epsilon}^*(Q_{\epsilon}(g(V_{\epsilon,k})-g(U_{\epsilon,k})))\|_X=o(\epsilon) \quad \text{and}\quad \|i^*_{\epsilon}(f(V_{\epsilon,k})-f(U_{\epsilon,k}))\|_X=o(\epsilon)  $$
and 
$$\|i_{\epsilon}^*((g'(V_{\epsilon,k})-g'(U_{\epsilon,k}))\phi)\|_X\leq o(1)\|\phi\|_X\quad \text{and}\quad \|i_{\epsilon}^*((f'(V_{\epsilon,k})-f'(U_{\epsilon,k}))\phi)\|_X\leq o(1)\|\phi\|_X$$ 
\end{lemma}

\begin{proof}[Proof of Lemma \ref{Control bounded}]

We only prove these results for $g$.

Firstly, $\Omega$ is bounded, then we can find $l>0$ such that $\Omega\subset B(0,l)$ then, $\Omega_{\epsilon}\subset B(0,\frac{l}{\epsilon})$ then $mes(\Omega_{\epsilon})\leq M\epsilon^{-n}$. Then, we have

\smallskip

\begin{itemize}
    \item If $\beta = \frac{2n}{n+2}$,

$$ \|g(V_{\epsilon,k})-g(U_{\epsilon,k})\|_\beta\leq C(\|(|V_{\epsilon,k}-U_{\epsilon,k}|)^q\|_\beta+\|U_{\epsilon,k}^{q-1}(V_{\epsilon,k}-U_{\epsilon,k})\|_\beta) $$$$\leq  C mes(\Omega_{\epsilon})^{\frac{1}{\beta}}(\|V_{\epsilon,k}-U_{\epsilon,k}\|_\infty^q+\|V_{\epsilon,k}-U_{\epsilon,k}\|_\infty)$$$$\leq C\epsilon^{-\frac{n}{\beta}}(\epsilon^{q\alpha}+\epsilon^{\alpha})\leq C'\epsilon^{\alpha-\frac{n}{\beta}}.$$

But, $\alpha -\frac{n}{\beta} =(n-2)-\frac{n+2}{2}=\frac{n}{2}-3>1$ because $n\geq9$. Then, $$\|g(V_{\epsilon,k})-g(U_{\epsilon,k})\|_\beta = o(\epsilon).$$

Furthermore, 
$$\|(g'(V_{\epsilon,k})-g'(U_{\epsilon,k}))\phi\|_\frac{2n}{n+2}\leq C\|\phi\|_X\|g'(V_{\epsilon,k})-g'(U_{\epsilon,k})\|_{L^\frac{n}{2}(\Omega_{\epsilon})}$$

But,  $$\|g'(V_{\epsilon,k})-g'(U_{\epsilon,k})\|_{L^\frac{n}{2}}\leq \|g'(V_{\epsilon,k})-g'(U_{\epsilon,k})\|_{\infty}mes(\Omega_\epsilon)^{\frac{2}{n}}$$

and

$$|g'(V_{\epsilon,k})-g'(U_{\epsilon,k})|\leq |V_{\epsilon,k}-U_{\epsilon,k}|^{q-1}\leq \|V_{\epsilon,k}-U_{\epsilon,k}\|_{L^\infty(\Omega_{\epsilon})}^{q-1} = o(\epsilon^2)$$
and
$$mes(\Omega_{\epsilon})\leq C diam(\Omega_{\epsilon})^{n}\leq C'\epsilon^{-n}$$

Conclusion: 
$$\|(g'(V_{\epsilon,k})-g'(U_{\epsilon,k}))\phi\|_\frac{2n}{n+2}=o(1)\|\phi\|_X$$

 $$\\$$

 \item If $\beta =\frac{ns}{n+2s}$, 
$$ \|(|V_{\epsilon,k}-U_{\epsilon,k}|)^q\|_\beta\leq \epsilon^{q\alpha-\frac{n}{\beta}}$$
In addition, because $\|U_{\epsilon,k}^{q-1}\|_\beta\ \leq  C\epsilon^{\alpha(q-1)-\frac{n}{\beta}}$, $$\|U_{\epsilon,k}^{q-1}(V_{\epsilon,k}-U_{\epsilon,k})\|_\beta\leq \|U_{\epsilon,k}^{q-1}\|_\beta\|V_{\epsilon,k}-U_{\epsilon,k}\|_\infty\leq C\epsilon^{(\alpha q)-\frac{n}{\beta}}$$

But, $\alpha q - \frac{n}{\beta}>1$  by hypothesis. 

Furthermore,
$$\|(g'(V_{\epsilon,k})-g'(U_{\epsilon,k}))\phi\|_\frac{ns}{n+2s}\leq C\|\phi\|_X\|g'(V_{\epsilon,k})-g'(U_{\epsilon,k})\|_{L^\frac{n}{2}(\Omega_{\epsilon})}=o(1)$$

\end{itemize}
$\\$

Conclusion, $\|i_{\epsilon}^*(Q_{\epsilon}(g(V_{\epsilon,k})-g(U_{\epsilon,k})))\|_X=o(\epsilon)$

\end{proof}

\begin{proof}[Proof of Lemma \ref{Stability3}]
In some sense, Lemma \ref{Control bounded} allows us to replace $V_{\epsilon,k}$ by $U_{\epsilon,k}$ in the expression of $T_{\epsilon,P}$. Thus, thanks to the same calculations as in the case of $\mathbb{R}^n$, we obtain

     $$\|N_f(\phi) +N_g(\phi)\|_X \leq C\|\phi\|_X^q$$ 
        $$\|\sum i_{\epsilon}^*((Q_{\epsilon} - Q(P_k))g(V_{\epsilon,k}))+I_f+I_g\|_X\leq C\epsilon$$
       $$ \sum \|i_{\epsilon}^*((Q_{\epsilon} - Q(P_k))g'(U_{\epsilon,k})\phi)\|_X+\|I_{f'}+I_{g'}\|_X=o(1)\|\phi\|_X$$

$\\$

\noindent Using \ref{Control bounded}, we easily deduce that 
$$\|\sum i_{\epsilon}^*(f(V_{\epsilon,k})-f(U_{\epsilon,k})-Q(P_k)(g(V_{\epsilon,k})-g(U_{\epsilon,k})))\|_X = o(\epsilon)$$

and 
 $$\|\sum i_{\epsilon}^*((f'(V_{\epsilon,k})-f'(U_{\epsilon,k})-[g'(V_{\epsilon,k})-g'(U_{\epsilon,k})]Q(P_k))\phi)\|_X=o(1)\|\phi\|_X.$$

In conclusion, $\|T_{\epsilon,P}(\phi)\|_X \leq C(\|\phi\|_X^q+ o(1)\|\phi\|_X+\epsilon)$ and we conclude easily. 
\end{proof}

\begin{lemma}\label{Contratction3}
$\\$

$T_{\epsilon,P}$ is a contraction on $B(0,N\epsilon)$ with a Lipschitz constant that does not depend on $P$.

$$\\$$

\end{lemma}
\begin{proof}
The same calculations in the unbounded case and the control of $R_2$ obtained in the last lemma led to the result.
\end{proof}

Finally as before, we deduce the existence of $\phi_{\epsilon,P} \in  B(0,N\epsilon)\cap X_{\epsilon}$ such that $T_{\epsilon,P}\phi_{\epsilon,P}=\phi_{\epsilon,P}$. In addition, $P \mapsto \phi_{\epsilon,P}$ is continuous.

\subsection{Solving the bifurcation equation}
 Fix $k \in \{1,2\}$ and $j \in \{1,...,n\}$. Define
$$G_{\epsilon,k,j}(P)=\langle -W_{\epsilon,P}  + i_{\epsilon}^*(f(W_{\epsilon,P} ) - g(W_{\epsilon,P} )Q_{\epsilon})|\psi_{\epsilon,k,j}\rangle $$Until now, we have worked with an arbitrary compact set $E$. Now introduce for $k\in\{1,2\}$, $R_k$ a real such that $Q_{|B(\xi_k,R_k)}$ has only $\xi_k$ as a critical point (it is possible because $\xi_k$ is a non degenerate critical point), such that $d(B(\xi_1,R_1),B(\xi_2,R_2))>0$ and $B(\xi_k,R_k)\subset \Omega$. Until the end we will consider $E=\bar{B}(\xi_1,R_1)\times\bar{B}(\xi_2,R_2) $. As before, we now prove following result
\begin{lemma}\label{Brouwer3}

There are constants $C$ and $l$ such that:
 $P \rightarrow (\frac{CG_{\epsilon,k,j}(P)}{Q(P_k)^l\epsilon})_{k,j}$ converges uniformly to $\nabla \tilde Q$ when $\epsilon \rightarrow0$ 
where $\tilde{Q}:(P_1,P_2) \rightarrow \sum_{k} Q(P_k)$.
$$\\$$
\end{lemma}
\begin{proof}
For the sake of simplicity, we will suppose that ${k_0}=1$.
$$G_{\epsilon,{k_0},j}(P)=\langle -W + i_{\epsilon}^*(f(W) - g(W)Q_{\epsilon})|\psi_{\epsilon,{k_0},j}\rangle  $$

$$=-\langle \phi|\psi_{\epsilon,{k_0},j}\rangle + \langle i_{\epsilon}^* ((f'(U_{\epsilon,{k_0}})-g'(U_{\epsilon,{k_0}})Q(P_{k_0}))\phi)|\psi_{\epsilon,{k_0},j}\rangle +\sum \langle i_{\epsilon}^*(g(V_{\epsilon,k})(Q(P_k)-Q_{\epsilon})|\psi_{\epsilon,{k_0},j}\rangle +$$
$$\langle I_f+I_g|\psi_{\epsilon,{k_0},j}\rangle+\langle N_f(\phi)|\psi_{\epsilon,{k_0},j}\rangle +    \langle N_g(\phi)|\psi_{\epsilon,{k_0},j}\rangle +\langle R_1|\psi_{\epsilon,{k_0},j}\rangle  $$$$+\langle (f'(S_{\epsilon,P})-f'(V_{\epsilon,{k_0}}))\phi|\psi_{\epsilon,{k_0},j}\rangle -\langle Q_{\epsilon}(g'(S_{\epsilon,P})-g'(V_{\epsilon,{k_0}}))\phi|\psi_{\epsilon,{k_0},j}\rangle $$$$+\langle i_{\epsilon}^*((f'(V_{\epsilon,{k_0}})-f'(U_{\epsilon,{k_0}})-[g'(V_{\epsilon,{k_0}})-g'(U_{\epsilon,{k_0}})]Q(P_{k_0}))\phi)|\psi_{\epsilon,{k_0},j}\rangle+\langle i^*(g'(V_{\epsilon,{k_0}})(Q(P_{k_0})-Q_{\epsilon})\phi|\psi_{\epsilon,{k_0},j}\rangle$$
Using Lemma \ref{Control bounded} and Section 3, we deduce easily that

$$|\langle N_f(\phi)+N_g(\phi)|\psi_{\epsilon,{k_0},j}\rangle|=o(\epsilon)$$
$$+\langle (f'(S_{\epsilon,P})-f'(V_{\epsilon,{k_0}}))\phi|\psi_{\epsilon,{k_0},j}\rangle -\langle Q_{\epsilon}(g'(S_{\epsilon,P})-g'(V_{\epsilon,{k_0}}))\phi|\psi_{\epsilon,{k_0},j}\rangle =o(\epsilon)$$

 $$\langle R_1|\psi_{\epsilon,{k_0},j}\rangle  +\langle i_{\epsilon}^*((f'(V_{\epsilon,{k_0}})-f'(U_{\epsilon,{k_0}})-[g'(V_{\epsilon,{k_0}})-g'(U_{\epsilon,{k_0}})]Q(P_{k_0}))\phi)|\psi_{\epsilon,{k_0},j}\rangle=o(\epsilon)$$
 $$\langle i^*(g'(V_{\epsilon,{k_0}})(Q(P_{k_0})-Q_{\epsilon})\phi|\psi_{\epsilon,{k_0},j}\rangle=o(\epsilon)$$
\medskip
\medskip
\par
 \noindent  \underline{Control of $-\langle \phi|\psi_{\epsilon,{k_0},j}\rangle + \langle i_{\epsilon}^* ((f'(U_{\epsilon,{k_0}})-g'(U_{\epsilon,{k_0}})Q(P_{k_0}))\phi)|\psi_{\epsilon,{k_0},j}\rangle$ 
 :}
 $\\$

\noindent One has, 
$$-\langle \phi|\psi_{\epsilon,{k_0},j}\rangle_{H^1_0}=\langle \phi|\Delta \psi_{\epsilon,{k_0},j}\rangle_{2}$$
But,

$$\langle i_{\epsilon}^* ((f'(U_{\epsilon,{k_0}})-g'(U_{\epsilon,{k_0}})Q(P_{k_0}))\phi)|\psi_{\epsilon,{k_0},j}\rangle_X=\langle(f'(U_{\epsilon,{k_0}})-g'(U_{\epsilon,k_0})Q(P_{k_0}))\phi|\psi_{\epsilon,{k_0},j}\rangle_2$$
$$=\left\langle(f'(U_{\epsilon,{k_0}})-g'(U_{\epsilon,{k_0}})Q(P_{k_0}))\phi|\psi_{\epsilon,{k_0},j}-\frac{\partial U_{\epsilon,{k_0}}}{\partial x_j}\right\rangle_2+\langle\phi|-\Delta \psi_{\epsilon,{k_0},j}\rangle_2$$
Then, by Hölder's inequality, 
$\\$

$$\left|\left|\left\langle(f'(U_{\epsilon,{k_0}})-g'(U_{\epsilon,{k_0}})Q(P_{k_0}))\phi|\psi_{\epsilon,{k_0},j}-\frac{\partial U_{\epsilon,{k_0}}}{\partial x_j}\right\rangle_2\right|\right|\leq \|\phi\|_2\left|\left|\psi_{\epsilon,{k_0},j}-\frac{\partial U_{\epsilon,{k_0}}}{\partial x_j}\right|\right|_2$$$$\leq \|\phi\|_X\left|\left|\psi_{\epsilon,{k_0},j}-\frac{\partial U_{\epsilon,{k_0}}}{\partial x_j}\right|\right|_{\infty}\epsilon^{-\frac{n}{2}}=o(\epsilon)$$
because, 
$\epsilon^{\alpha+1-\frac{n}{2}}=o(\epsilon)$.
In conclusion, 

$$-\langle \phi|\psi_{\epsilon,{k_0},j}\rangle + \langle i_{\epsilon}^* ((f'(U_{\epsilon,{k_0}})-g'(U_{\epsilon,{k_0}})Q(P_{k_0}))\phi)|\psi_{\epsilon,{k_0},j}\rangle = o(\epsilon)$$
\medskip
\par

 \noindent  \underline{Control of $\sum\langle i^*(g(V_{\epsilon,k})(Q(P_{k})-Q_{\epsilon}))|\psi_{\epsilon,{k_0},j}\rangle$}:

\medskip

On the one hand,
$$|\langle i_{\epsilon}^*(g(V_{\epsilon,2})(Q_{\epsilon}-Q(P_2))|\psi_{\epsilon,{k_0},j}\rangle_X|=|\langle g(V_{\epsilon,2})(Q_{\epsilon}-Q(P_2))|\psi_{\epsilon,{k_0},j}\rangle_2|  =o(\epsilon) + \left|\left\langle g(U_{\epsilon,2})(Q_{\epsilon}-Q(P_2))|\frac{\partial U_{\epsilon,1}}{\partial x_j}\right\rangle_2\right|$$
But, as before 
$\left|\left\langle g(U_{\epsilon,2})(Q_{\epsilon}-Q(P_2)|\frac{\partial U_{\epsilon,1}}{\partial x_j}\right\rangle_2\right|=o(\epsilon)$
Thus, 
$$|\langle i_{\epsilon}^*(g(V_{\epsilon,2})(Q_{\epsilon}-Q(P_2))|\psi_{\epsilon,k_0,j}\rangle|=o(\epsilon)$$
On the other hand
$$\langle i^*(g(V_{\epsilon,1})(Q(P_1)-Q_{\epsilon})|\psi_{\epsilon,{k_0},j}\rangle_X=\left\langle g(U_{\epsilon,1})(Q(P_1)-Q_{\epsilon})|\frac{\partial U_{\epsilon,1}}{\partial x_j }\right \rangle+o(\epsilon)$$

Then, by the Stokes Theorem
$$\left\langle g(U_{\epsilon,1})(Q(P_1)-Q_{\epsilon})|\frac{\partial U_{\epsilon,1}}{\partial x_j }\right \rangle= \frac{\epsilon}{q+1}\int_{\Omega_{\epsilon}-\frac{P_1}{\epsilon}}\frac{\partial Q(\epsilon x+P_1)}{\partial x_j}U^{q+1}_{P_1}(x)dx +\frac{1}{q+1}\int_{\partial \Omega_{\epsilon}} U^{q+1}_{\epsilon,1}v_j(Q(P_1)-Q_{\epsilon})$$

Then, 
$$\int_{\Omega_{\epsilon}-\frac{P_1}{\epsilon}}\frac{\partial Q(\epsilon x+P_1)}{\partial x_j}U^{q+1}_{P_1}(x)dx=\int_{\Omega_{\epsilon}-\frac{P_1}{\epsilon}}\frac{\partial Q(P_1)}{\partial x_j}U^{q+1}_{P_1}(x)+o(1)=\int_{\mathbb{R}^n}\frac{\partial Q(P_1)}{\partial x_j}U^{q+1}_{P_1}(x)+o(1)$$
In addition, 

$$\left|\int_{\partial \Omega_{\epsilon}} U^{q+1}_{\epsilon,1}v_j(Q(P_1)-Q_{\epsilon})\right|\leq CH^{n-1}(\partial\Omega_{\epsilon})\epsilon^{\alpha(q+1)}=o(\epsilon)$$

$$\left\langle g(U_{\epsilon,1})(Q(P_1)-Q_{\epsilon})|\frac{\partial U_{\epsilon,1}}{\partial x_j }\right \rangle=C\epsilon\frac{\partial Q(P_1)}{\partial x_j}Q(P_1)^{(q+1)a-nb}\int_{\mathbb{R}^n}U^{q+1}+o(\epsilon)$$
We conclude that $$\left(\frac{G_{\epsilon}(P)}{Q(P_k)^l\epsilon}\right)_{k,j}$$ converges uniformly to $C\nabla \tilde{Q}$ which concludes.
\end{proof}

$\\$$\\$
Using Lemma \ref{Brouwer2} and \ref{solution} we deduce 

\begin{proposition}\label{end unbounded}
There is a family $(P_{\epsilon})$ of points of $\Omega$ such that 
\begin{itemize}
    \item $P_{\epsilon,k}\rightarrow \xi_k$

    \item If $ W_{\epsilon}:=\sum V_{\epsilon,P_{\epsilon}}+\phi_{\epsilon,P_{\epsilon}}$, then 
$$W_{\epsilon}=i_{\epsilon}^*(f(W_\epsilon)-Q_{\epsilon}g(W_{\epsilon})).$$
\end{itemize}

\end{proposition}

\subsection{Properties of the solution}
As before, we can deduce some properties on $W$. 
\begin{lemma}
For all $0<\epsilon<\epsilon_0$, $W_{\epsilon,P_{\epsilon}}\in C^2(\bar{\Omega}_{\epsilon})$ and $W_{\epsilon,P_{\epsilon}}$ is a strong solution of
$$-\Delta u = u^p-Q_{\epsilon}u^q $$
\end{lemma}
\begin{proof}
Apply directly Theorem 1.16 of \cite{AmbrosettiMalchiodi2007}.
\end{proof}

Now we have a strong solution $W_{\epsilon}$ of \eqref{BVP4}. We want to show that this solution is positive to answer to the initial problem and respects the required properties (the proofs are similar as before),

\begin{proposition}\label{positivity}
    If $\epsilon >0$ is small enough, for all $0<\epsilon <\epsilon_0$, $W_{\epsilon,P_{\epsilon}}>0$.

Now we set $(u_{\epsilon})=W_{\epsilon}(\frac{x}{\epsilon})$.
 $u_{\epsilon}$ concentrates at $(P_{\epsilon})$
and we have for $R>0$ and $\epsilon$ small enough, $$sup_{y\notin B(P_{\epsilon,1},R\epsilon)\cup B(P_{\epsilon,2},R\epsilon)} |u_{\epsilon}(y)|\leq \frac{C}{R^\alpha}+C\epsilon.$$

\end{proposition}
\begin{proof}
The proof of the positivity is the same. We only need to prove that $W_{\epsilon}$ concentrates. Denote $\tilde{W}_{\epsilon}$ the extension of $W_\epsilon$ on $\mathbb{R}^n$. $\tilde{W}$ is a subsolution of the equation
$$-\Delta u=\tilde{W}^p-Q_{\epsilon}\tilde{W}^q. $$
(We extend $Q$ by a $C^2$ function on $\mathbb{R}^n$). Indeed, take $\eta \in C_c^\infty(\mathbb{R}^n)$ such that $\eta \geq 0$. We have 

$$\int_{\mathbb{R}^n}\nabla\tilde{W}\nabla\eta =\int_{\Omega_{\epsilon}} (W^p-Q_\epsilon W^q)\eta +\int_{\partial \Omega_{\epsilon}} \frac{\partial W}{\partial v}\eta $$

But, $W>0$ on $\Omega_{\epsilon}$ and $W=0$ on $\partial\Omega_{\epsilon}$. Then $\frac{\partial W}{\partial v}\leq 0$. Thus, 
$$\int_{\mathbb{R}^n}\nabla\tilde{W}\nabla\eta \leq \int_{\mathbb{R}^n} (\tilde{W}^p-Q_\epsilon \tilde{W}^q)\eta. $$

We conclude applying Theorem 4.1 of \cite{HanLin2011}.

\end{proof}

\newpage 
$\\$$\\$

\end{document}